\documentclass[11pt, twoside]{article}

\usepackage[english]{babel}

\usepackage{amssymb}
\usepackage{mathrsfs}
\usepackage{amsmath}
\usepackage{amsthm}
\usepackage{amsfonts}
\usepackage{latexsym}
\usepackage{indentfirst}
\usepackage{color}
\usepackage{txfonts}
\usepackage{enumerate}

\usepackage{anysize}

\usepackage[colorlinks=true,
linkcolor=blue,
citecolor=red,
urlcolor=magenta,
backref=page
]{hyperref}

\allowdisplaybreaks

\hypersetup{pdftitle={Sharp Truncated Difference Estimates
for Nonlocal Nonconvex Functionals},
pdfauthor={Feng Dai, Dachun Yang, Wen Yuan and Yirui Zhao}}

\newtheorem{theorem}{Theorem}[section]
\newtheorem{lemma}[theorem]{Lemma}
\newtheorem{corollary}[theorem]{Corollary}
\newtheorem{proposition}[theorem]{Proposition}
\theoremstyle{definition}
\newtheorem{remark}[theorem]{Remark}

\numberwithin{equation}{section}

\begin{document}
\title{\bf\Large Sharp Truncated Difference Estimates for Nonlocal Nonconvex Functionals
Defined by Superlevel Sets\footnotetext{\hspace{-0.35cm}2020 
\emph{Mathematics Subject Classification}.
Primary 26D10; Secondary 35A23, 46E35, 26B30.
\endgraf\emph{Key words and phrases}.
Poincar\'e inequality, nonlocal nonconvex functional, Sobolev space, bounded variation,
compactness, vanishing mean oscillation.
\endgraf
This project is partially supported by the National
Natural Science Foundation of China
(Grant Nos. 12431006, 12371093, and 125B2008),
the Beijing Natural Science Foundation (Grant No. 1262011),
and the Fundamental Research Funds
for the Central Universities (Grant No. 2253200028).
The first author is also supported by the
NSERC of Canada Discovery Grant (Grant No. RGPIN 2026-05667).}}
\author{Feng Dai, Dachun Yang,
Wen Yuan and Yirui Zhao}
\maketitle
\vspace{-0.7cm}
\begin{center}
\begin{minipage}{13cm}
{\small {\bf Abstract.}\quad
Let $N\in\mathbb N$ and $p\in[1,\infty)$.
In this article, we consider a family
$\{\Phi_{\lambda,p}^{\gamma}\}_{\lambda\in(0,\infty)}$
of nonlocal nonconvex functionals
with a real parameter $\gamma$
and establish two truncated difference estimates bounded by these functionals
with the sharp range of $\gamma$; that is, for any $\gamma\in(-\infty,-1]$,
any $\lambda\in(0,\infty)$, any cube $Q\subset\mathbb R^N$,
and any measurable function $g:Q\to\mathbb R$,
\begin{align*}
&\int_Q\int_Q
\left[\frac{|g(x)-g(y)|-\lambda|x-y|^{1+\gamma/p}}{\ell(Q)^{1+\gamma/p}}\right]_+^p
\,dx\,dy\\
&\quad\lesssim \lambda^p
\mathop{\int_Q\int_Q}\limits_{
\frac{|g(x)-g(y)|}{|x-y|^{1+\gamma/p}}>\lambda}
\left[\frac{|x-y|}{\ell(Q)}\right]^{\gamma-N}\,dx\,dy
=:\frac{\Phi_{\lambda,p}^{\gamma}(g;Q)}{\ell(Q)^{\gamma-N}},
\end{align*}
where the implicit positive
constant depends only on $N$, $p$, and $\gamma$.
For $\gamma\in(0,\infty)$, the above estimate holds with the
left-most integrand replaced by
$\frac{\min\{|g(x)-g(y)|^p,\lambda^p|x-y|^{p+\gamma}\}}{\ell(Q)^{p+\gamma}}$.
As applications, these estimates yield related nonlocal Poincar\'e inequalities
with the sharp range of $\gamma$, and also recover the known $\Gamma$-liminf lower bounds and
characterizations of Sobolev and BV spaces in terms of
the liminf of the functionals $\{\Phi^\gamma_{\lambda,p}\}_{\lambda\in(0,\infty)}$
for all $\gamma\in\mathbb{R}\setminus(-1,0]$. We investigate
further applications to compactness and variational
problems as well as a new VMO criterion.
Inspired by a recent work of R. L. Frank, we develop an
elementary approach based on
dyadic iteration of truncated differences along line segments,
which allows us to treat the aforementioned problems in a unified way.
}
\end{minipage}
\end{center}
\vspace{0.1cm}

\tableofcontents

\section{Introduction}

The Poincar\'e inequality plays a fundamental role in the study of
Sobolev spaces and variational problems. On a bounded connected Lipschitz
domain $\Omega\subset\mathbb R^N$, it states that, for any
$p\in[1,\infty)$ and $u\in W^{1,p}(\Omega)$,
\begin{equation}\label{intro:eq:classical-poincare}
\int_\Omega|u-u_\Omega|^p\,dx
\lesssim\int_\Omega|\nabla u|^p\,dx,
\end{equation}
where $u_\Omega:=\frac{1}{|\Omega|}\int_\Omega u\,dx$,
$W^{1,p}(\Omega)$ denotes the Sobolev space of functions in
$L^p(\Omega)$ whose distributional gradient belongs to
$L^p(\Omega;\mathbb R^N)$, and the implicit positive constant depends
only on $N$, $p$, and $\Omega$; see, for instance,
\cite[Theorem~13.27]{Leoni}.

Ponce \cite[Theorem~1.1]{Ponce2004} established a variant of
the Poincar\'e inequality
based on the famous formulae of Bourgain et al.
\cite{BBM2001}.
Specifically, when $N\geq2$ and $\rho_n$ are nonnegative radial kernels
of unit mass concentrating at the origin, then, for any $n\geq n_0$ and
$u\in L^p(\Omega)$,
\begin{equation}\label{intro:eq:ponce-poincare}
\int_\Omega|u-u_\Omega|^p\,dx
\lesssim\iint_{\Omega\times\Omega}
\frac{|u(x)-u(y)|^p}{|x-y|^p}\rho_n(x-y)\,dx\,dy,
\end{equation}
where the implicit positive constant depends
only on $N$, $p$, and $\Omega$, and where $n_0$ may also depend on $\rho_n$.
Special cases of \eqref{intro:eq:ponce-poincare} were used by Bourgain et al. \cite{BBM2004} in their analysis of the
Ginzburg--Landau model, particularly in identifying the limit of
minimizers.

A different approximation of the gradient energy was introduced by
Nguyen \cite{Nguyen2006} and further studied by Bourgain and Nguyen
\cite{BourgainNguyen2006}. For $p\in[1,\infty)$, $\lambda\in(0,\infty)$,
and a measurable function $u$ on an open set
$U\subset\mathbb R^N$, consider
$$
I_{\lambda,p}(u;U)
:=\lambda^p\iint_{\{(s,t)\in U\times U:\,|u(s)-u(t)|>\lambda\}}
\frac{dx\,dy}{|x-y|^{N+p}}.
$$
For $p\in(1,\infty)$, Nguyen \cite[Theorem~2]{Nguyen2006} proved that
$I_{\lambda,p}(u;\mathbb R^N)$ converges, as $\lambda\to0^+$, to a
constant multiple of the Sobolev energy for $u\in W^{1,p}(\mathbb R^N)$.
Bourgain and Nguyen \cite{BourgainNguyen2006} characterized Sobolev spaces
among functions in $L^p(\mathbb R^N)$ by the condition
$\varliminf_{\lambda\to0^+}I_{\lambda,p}(u;\mathbb R^N)<\infty$,
and further characterizations can also be found in Nguyen \cite{Nguyen2008}.
Here, and thereafter, $\varliminf$ and $\varlimsup$ denote the \emph{limit
inferior} and the \emph{limit superior}, respectively.

Following these results, Nguyen \cite[Theorem~1]{Nguyen2011} established
a variant of the Poincar\'e inequality involving $I_{\lambda,p}$. Nguyen showed that,
for any given $p\in[1,\infty)$ and for any $\lambda\in(0,\infty)$,
any ball $B\subset\mathbb R^N$, and any measurable $u:B\to\mathbb R$,
\begin{equation}\label{intro:eq:nguyen-poincare}
\fint_B\fint_B|u(x)-u(y)|^p\,dx\,dy
\lesssim\lambda^p+r_B^{p-N}I_{\lambda,p}(u;B),
\end{equation}
where $\fint_B$ denotes the integral average over $B$, $r_B$
denotes the radius of $B$, and the implicit
positive constant depends only on $N$ and $p$.

Note that the functional $I_{\lambda,p}$ is nonconvex, unlike the functionals in
\eqref{intro:eq:ponce-poincare}. This creates a basic difficulty for associated minimization
problems because the usual argument based on weak compactness and
convexity is no longer available. Brezis and Nguyen
\cite{BrezisNguyenImages} investigated the existence of minimizers and
the $\Gamma$-convergence for a broader class of nonlocal nonconvex
approximations of the total variation arising in image denoising.

The Poincar\'e inequalities \eqref{intro:eq:ponce-poincare} and
\eqref{intro:eq:nguyen-poincare} provide a bridge between
mean oscillations and nonlocal energies,
without requiring Sobolev smoothness in advance.
Thus, they may serve as alternatives to the classical Poincar\'e inequalities
on metric measure spaces.
For further study on these inequalities,
we refer to
\cite{BrezisNguyenVMO,BrezisNguyenImages,MallickNguyen2020,MallickNguyenCoulomb2022,Nguyen2011,NguyenDegree2017,
NguyenSquassinaLog2017,NguyenSquassina2019,VanSchaftingenGap2020,VanSchaftingenLifting2025}.

Recently, Brezis et al.\ \cite{BSVY,BVY2021} obtained new weak-type
representations of Sobolev seminorms by introducing a family of nonlocal
functionals with a real parameter $\gamma$.
To describe these functionals, for $\lambda\in(0,\infty)$,
$b\in\mathbb R$, and
a measurable function $u:U\to\mathbb R$ on a measurable set $U\subset\mathbb R^N$, let
$$
E_{\lambda,b}[u;U]
:=\left\{(x,y)\in U\times U:x\ne y,\ \frac{|u(x)-u(y)|}{|x-y|^{1+b}}>\lambda\right\}.
$$
For $p\in[1,\infty)$ and $\gamma\in\mathbb R$, define
\begin{equation}\label{app:eq:energy}
\Phi_{\lambda,p}^\gamma(u;U)
:=\lambda^p\iint\nolimits_{E_{\lambda,\frac{\gamma}{p}}[u;U]}
|x-y|^{\gamma-N}\,dx\,dy.
\end{equation}
Brezis et al.\ \cite{BSVY} proved
that, for any $\gamma\in\mathbb R\setminus\{0\}$ and
$u\in C_{\rm c}^\infty(\mathbb R^N)$,
\begin{equation*}
\lim_{\lambda\to L_\gamma}\Phi_{\lambda,p}^\gamma(u;\mathbb R^N)
=\frac{K_{N,p}}{|\gamma|}\int_{\mathbb R^N}|\nabla u|^p\,dx,
\end{equation*}
where
\begin{align}\label{1.4x}
K_{N,p}:=\int_{\mathbb S^{N-1}}|\omega_1|^p\,d\sigma(\omega),
\end{align}
with $\sigma$ denoting surface measure on the unit sphere $\mathbb S^{N-1}$,
and where $\lambda\to L_\gamma$ means $\lambda\to0^+$ when
$\gamma\in(-\infty,0)$ and $\lambda\to\infty$ when $\gamma\in(0,\infty)$.
For smooth functions, the preceding limit formula when $\gamma=-p$ and
$p\in(1,\infty)$ reduces to \cite[Theorem~2(c)]{Nguyen2006},
and when $\gamma=N$ to the formula in \cite[Theorem~1.2]{BVY2021}.
Brezis et al.\ \cite[Section~7B]{BSVY} further asked whether
$\varliminf_{\lambda\to L_\gamma}\Phi_{\lambda,p}^\gamma(u;\mathbb R^N)$
controls the Sobolev energy or total variation for any $u\in L^1_{\rm loc}(\mathbb R^N)$.

For some developments related to this question,
we refer to
\cite{AntonucciEtAl,Frank2024,GobbinoPicenni,Lahti2025,LahtiLi2026,LinYangYangYuanZhang,LinEtAlFullRange2026,Picenni2024}.
For further applications involving these families of nonlocal
functionals, we refer to
\cite{BVYLorentz2021,DaiEtAlBall2023,DominguezEtAlBesov2023,
Poliakovsky2022,HuEtAlHigherOrder2025,LiEtAlWeighted2026,DaiEtAlMetric2022,
HytonenYangYuanZhao2026,DaiEtAlOperators2026,DominguezMilman}.

In his recent survey, Nguyen \cite[Open Question~33]{NguyenSurvey}
asked whether \eqref{intro:eq:nguyen-poincare} extends to the general
family $\Phi_{\lambda,p}^\gamma$ under a positivity assumption on its
cell constant. For $Q_0:=(0,1)^N$ and
$\gamma\in\mathbb R\setminus\{0\}$, this \emph{cell constant} is defined as
\begin{equation}\label{eq:cell-constant}
\kappa_{N,p,\gamma}
:=\inf_{\{h_\lambda\}}
\left[\varliminf_{\lambda\to L_\gamma}
\Phi_{\lambda,p}^\gamma(h_\lambda;Q_0)\right],
\end{equation}
where the infimum is taken over all families of measurable functions
$h_\lambda:Q_0\to\mathbb R$ converging in measure on $Q_0$ to
$h_0(x):=\frac{x_1+\cdots+x_N}{\sqrt N}$ as $\lambda\to L_\gamma$;
see \cite[Section~8, (94)]{NguyenSurvey}.
The cell constant $\kappa_{N,p,\gamma}$ is the conjectured coefficient
of the local energy in the $\Gamma$-limit of \eqref{app:eq:energy}; see
\cite[Conjecture~28]{NguyenSurvey}.

In this article, we study the above two questions of Brezis et al.\ and Nguyen
in a unified framework based on the following truncated difference estimates.

\begin{theorem}\label{thm:truncated-differences}
Let $N\in\mathbb N$ and $p\in[1,\infty)$, and let $Q\subset\mathbb R^N$
be a cube. For any $\lambda\in(0,\infty)$ and any
measurable function $g:Q\to\mathbb R$, the following statements hold.
\begin{enumerate}[\rm(i)]
\item If $\gamma\in(-\infty,-1]$, then
\begin{equation}\label{eq:truncation-negative}
\fint_Q\fint_Q
\left(|g(x)-g(y)|-\lambda|x-y|^{1+\frac{\gamma}{p}}\right)_+^p
\,dx\,dy
\lesssim\ell(Q)^{p-N}\Phi_{\lambda,p}^{\gamma}(g;Q).
\end{equation}
\item If $\gamma\in(0,\infty)$, then
\begin{equation}\label{eq:truncation-positive}
\fint_Q\fint_Q
\min\left\{|g(x)-g(y)|^p,\lambda^p|x-y|^{p+\gamma}\right\}
\,dx\,dy
\lesssim\ell(Q)^{p-N}\Phi_{\lambda,p}^{\gamma}(g;Q).
\end{equation}
\end{enumerate}
Here, $\ell(Q)$ is the edge length of $Q$, and the implicit positive
constants depend only on $N$, $p$, and $\gamma$.
\end{theorem}

\begin{remark}\label{rem:truncation-sharpness}
The ranges of $\gamma$ in Theorem~\ref{thm:truncated-differences}
are sharp: \eqref{eq:truncation-negative} fails for
$\gamma\in(-1,\infty)$, and \eqref{eq:truncation-positive} fails for
$\gamma\in(-\infty,0]$, even on the unit cube.
The proof is given in Subsection~\ref{sec:failure}.
\end{remark}

The quantities on the left-hand sides of \eqref{eq:truncation-negative}
and \eqref{eq:truncation-positive} can be understood intuitively through the identity
$$
(s-a)_++\min\{s,a\}=s,\quad s,a\in[0,\infty).
$$
The two complementary estimates control the excess of the difference
above the threshold when $\gamma\in(-\infty,-1]$ and the difference
capped at the threshold when $\gamma\in(0,\infty)$.
When $\lambda\to L_\gamma$, both truncations recover the full difference.

As applications,
Theorem~\ref{thm:truncated-differences} gives the nonlocal Poincar\'e inequality
and the $\Gamma$-liminf lower bound in a relatively transparent way.
We present these details in Sections \ref{sec:main-theorem-completion}
and \ref{sec:gamma-applications}.
We first record the complete classification for the Poincar\'e inequality
involving $\Phi^\gamma_{\lambda,p}$.

\begin{corollary}\label{thm:classification}
Let $N\in\mathbb N$, $p\in[1,\infty)$, and $\gamma\in\mathbb R$,
and let $Q\subset\mathbb R^N$ be a cube.
For $\lambda\in(0,\infty)$ and a measurable function
$g:Q\to\mathbb R$, consider the inequality
\begin{equation}\label{1-1}
\fint_Q\fint_Q|g(x)-g(y)|^p\,dx\,dy
\leq c_{N,p,\gamma}\left[\lambda^p\ell(Q)^{p+\gamma}
+\ell(Q)^{p-N}\Phi_{\lambda,p}^\gamma(g;Q)\right].
\end{equation}
Then the following assertions hold.

\begin{enumerate}[\rm(i)]
\item If $\gamma\in(-1,\infty)$,  then there is no finite constant
$c_{N,p,\gamma}$ for which \eqref{1-1} holds simultaneously for all
$\lambda\in(0,\infty)$ and all measurable $g:Q\to\mathbb R$.
Moreover, if $\gamma\in(0,\infty)$ and $g\in L^p(Q)$, then
\eqref{1-1} holds with a finite constant independent of
$\lambda\in(0,\infty)$ if and only if $g$ is constant almost
everywhere on $Q$.

\item If   $\gamma\in (-\infty, -1]$, then there exists a constant
$c_{N,p,\gamma}\in(0,\infty)$ depending only on $N$, $p$, and $\gamma$
such that \eqref{1-1} holds for any $\lambda\in(0,\infty)$,
any cube $Q\subset\mathbb R^N$, and any measurable $g:Q\to\mathbb R$.

\item For any $M\in(1,\infty)$ and  $\gamma\in[-M,-1]$, the least  constant $c_{N,p,\gamma}$ in  \eqref{1-1} satisfies $$
 \frac{2\lvert \gamma\rvert }{(p+1)(p+2)K_{N,p}}\leq c_{N,p,\gamma}\leq c_{N,p,M}<\infty,
$$
where $K_{N,p}$ is defined as above and $c_{N,p,M}\in(0,\infty)$
depends only on $N$, $p$, and $M$.
\item For any $\gamma\in(-\infty,-1]$,
$0<\kappa_{N,p,\gamma}\leq \frac{K_{N,p}}{\lvert \gamma\rvert }$,
where $\kappa_{N,p,\gamma}$ is as in \eqref{eq:cell-constant}.
\end{enumerate}
\end{corollary}

Corollary~\ref{thm:classification} gives an affirmative answer to the
aforementioned question posed by
Nguyen  \cite[Open Question~33]{NguyenSurvey} for $\gamma\in(-\infty,-1]$ without assuming the
positivity of the cell constant, and shows that \eqref{1-1} fails in general
for $\gamma\in(-1,\infty)$.

Here are further comments on Corollary~\ref{thm:classification}.

\begin{remark}
\begin{enumerate}[\rm(i)]
\item Although the constant in \cite[Open Question~33]{NguyenSurvey}
is denoted by $C_{N,p}$, its dependence on $\gamma$ is necessary
by Corollary~\ref{thm:classification}(iii).

\item
When $\gamma=-p$, \eqref{1-1} reduces to Nguyen's Poincar\'e inequality
\eqref{intro:eq:nguyen-poincare}.

\item Let $p\in[1,\infty)$ and $\gamma\in(-\infty,-1]$.
Letting $\lambda\to0^+$ in \eqref{1-1} and applying
\eqref{app:eq:smooth-local-limit}, we can obtain
the classical Poincar\'e inequality \eqref{intro:eq:classical-poincare} on cubes.
In this sense, \eqref{1-1} is an improved Poincar\'e inequality.
\end{enumerate}
\end{remark}

Using the Poincar\'e inequality \eqref{1-1}, we obtain the compactness
at a fixed threshold in the sharp parameter range
(see Theorem~\ref{app:thm:fixed-compactness}).
This compactness yields the existence of minimizers for the associated
nonlocal variational problems (see Corollary~\ref{app:cor:regularization}).
When $p=N$, we also deduce from \eqref{1-1} a VMO criterion involving the condition
$\Phi_{\lambda,N}^{\gamma}(u;Q)<\infty$ for some $\lambda\in(0,\infty)$
(see Theorem~\ref{app:thm:single-threshold-vmo}).

A second consequence of Theorem~\ref{thm:truncated-differences}
is a $\Gamma$-liminf estimate.
For $u\in L^1_{\mathrm{loc}}(\Omega)$ on an open set
$\Omega\subset\mathbb R^N$, write
\begin{equation*}
\mathcal E_p(u;\Omega):=
\begin{cases}
\displaystyle\int_\Omega|\nabla u|^p\,dx
&\text{if }p\in(1,\infty)\text{ and }u\in\dot W^{1,p}(\Omega),\\
|Du|(\Omega)&\text{if }p=1\text{ and }u\in\dot{\mathrm{BV}}(\Omega),\\
\infty&\text{otherwise}.
\end{cases}
\end{equation*}
Here, $\dot W^{1,p}(\Omega)$ and $\dot{\mathrm{BV}}(\Omega)$ consist
of locally integrable functions whose distributional gradients belong
to $L^p(\Omega;\mathbb R^N)$ and to the space of finite
vector-valued Radon measures on $\Omega$, respectively.

\begin{corollary}\label{app:thm:gamma-coercivity}
Let $N\in\mathbb N$, $p\in[1,\infty)$, and
$\gamma\in\mathbb R\setminus(-1,0]$, and let
$\Omega\subset\mathbb R^N$ be an open set.
If $\{\lambda_j\}_{j\in\mathbb N}\subset(0,\infty)$ and
$\{u_j\}_{j\in\mathbb N}\subset L^1_{\mathrm{loc}}(\Omega)$ satisfy
$\lambda_j\to L_\gamma$ and $u_j\to u$ in
$L^1_{\mathrm{loc}}(\Omega)$, then
$$
\varliminf_{j\to\infty}\Phi_{\lambda_j,p}^{\gamma}(u_j;\Omega)\ge c_{N,p,\gamma}
\mathcal E_p(u;\Omega),
$$
where the constant $c_{N,p,\gamma}\in(0,\infty)$ depends only on $N$, $p$, and
$\gamma$. Equivalently,
\begin{equation*}
\Gamma\text{-}\varliminf_{\lambda\to L_\gamma}
\Phi_{\lambda,p}^{\gamma}(\cdot;\Omega)
\geq c_{N,p,\gamma}\mathcal E_p(\cdot;\Omega)
\quad\text{in }L^1_{\mathrm{loc}}(\Omega).
\end{equation*}
\end{corollary}

For $p\in[1,\infty)$, $\gamma\in\mathbb R\setminus(-1,0]$, and
$u\in L^1_{\mathrm{loc}}(\mathbb R^N)$, applying
Corollary~\ref{app:thm:gamma-coercivity} to a constant sequence, we obtain
\begin{equation}\label{intro:eq:liminf-lower-bound}
\varliminf_{\lambda\to L_\gamma}\Phi_{\lambda,p}^\gamma(u;\mathbb R^N)
\ge c_{N,p,\gamma}\mathcal E_p(u;\mathbb R^N).
\end{equation}
This inequality is known. Specifically,
for $\gamma\in[-p,-1]$, the liminf estimate \eqref{intro:eq:liminf-lower-bound}
is recorded in \cite[Theorem~16]{NguyenSurvey}.
For $\gamma\in(-\infty,-p)$, Picenni
\cite[Theorem~1.1 and Corollary~1.2]{Picenni2026} proved the
$\Gamma$-liminf lower bound, which yields the
liminf characterization.
For $\gamma\in(0,\infty)$, Gobbino and Picenni
\cite[Theorem~1.1]{GobbinoPicenni} established the $\Gamma$-liminf
lower bound in $L^1_{\mathrm{loc}}(\mathbb R^N)$, and Lin et al.\ \cite[Theorem~1.1]{LinYangYangYuanZhang} obtained the full
$\Gamma$-limit in $L^p(\mathbb R^N)$ on bounded Lipschitz domains.
For $\gamma\in(-\infty,-1]$, the full $\Gamma$-convergence in $L^p(\mathbb R^N)$
on bounded Lipschitz domains was established in
\cite[Theorem~1.3(ii)]{LinEtAlFullRange2026}.
Lin et al.\ \cite[Theorem~1.4]{LinEtAlFullRange2026} also
obtained the full $\Gamma$-convergence in
$L^1_{\mathrm{loc}}(\mathbb R^N)$ for both negative and positive
$\gamma$.
Our purpose here is to give a unified and elementary proof for
both negative and positive $\gamma$ as a direct consequence of Theorem~\ref{thm:truncated-differences}.
Our proof works in any Euclidean space $\mathbb R^N$, without reducing the problem to
$\mathbb R$ or invoking the estimates for the ``asymptotic cost of
oscillations'' in \cite[Propositions~2.2 and~3.1]{GobbinoPicenni}.

The proof of Theorem~\ref{thm:truncated-differences} combines the
dyadic iteration for differences along line segments  with
the ideas of using finiteness of the
nonlocal energy to control the remainder asymptotically from Frank
\cite[Lemma~2]{Frank2026}.
The approach developed in this article
gives unified and elementary proofs of the
liminf characterizations of Bourgain and Nguyen
\cite{BourgainNguyen2006},
the Poincar\'e inequality of Nguyen \cite{Nguyen2011} for $\gamma=-p$,
and the $\Gamma$-liminf lower bound of Gobbino and Picenni
\cite{GobbinoPicenni} for $\gamma\in(0,\infty)$.
Brezis and Nguyen \cite[Remark~2]{BrezisNguyenVMO} called for a more
elementary proof of Nguyen's Poincar\'e inequality on cubes, even in
dimension one. The proof presented here offers an elementary approach to
this inequality, and hence answers their question;
see Remark~\ref{rem:elementary-nguyen-proof}.

The remainder of this article is organized as follows.
In Section~\ref{sec:truncation-estimates}, we show that both estimates in
Theorem~\ref{thm:truncated-differences} hold.
In Section~\ref{sec:main-theorem-completion}, we prove the Poincar\'e
inequality \eqref{1-1}, its failure for $\gamma\in(-1,\infty)$, the dependence
of its constant on $\gamma$, and the positivity of the cell constant,
completing the proof of Corollary~\ref{thm:classification}.
Section~\ref{sec:gamma-applications} presents a unified proof of the
known $\Gamma$-liminf lower bounds and liminf characterizations of
Sobolev and BV spaces.
The applications to compactness, variational problems, and VMO are
presented in Section~\ref{sec:applications}.

We end this introduction by making some notational conventions.
Let $\mathbb{N}:=\{1,2,\ldots\}$.
Throughout this article, cubes have edges parallel to the coordinate axes.
For a measurable set $A$, $|A|$ denotes its Lebesgue measure.
Given a set $E\subset\mathbb{R}^N$, we denote its characteristic
function by $\mathbf{1}_E$.
For an open set $U$, $W^{1,p}(U)$ consists of the functions in $L^p(U)$
whose distributional gradient belongs to $L^p(U;\mathbb R^N)$, and
$\mathrm{BV}(U)$ consists of the functions in $L^1(U)$ whose distributional
gradient is a finite vector-valued Radon measure, with total variation
$|Du|(U)$. The notation $f\lesssim g$ means that $f\leq Cg$ for some
constant $C\in(0,\infty)$. If both $f\lesssim g$ and $g\lesssim f$
hold, we write $f\sim g$. The notation $\lambda\to 0^+$ means that
there exists a positive constant $c_0$ such that $\lambda\in (0,c_0)$
and $\lambda\to 0$. Finally, in all proofs we consistently retain the notation
introduced in the original theorem (or related statement).

\section{Proof of Theorem~\ref{thm:truncated-differences}}
\label{sec:truncation-estimates}

In this section, we show that the two complementary estimates in
Theorem~\ref{thm:truncated-differences} hold.
We first record three elementary lemmas that will be used in both
negative and positive $\gamma$. The first one extends the summation
argument in \cite[(10)]{Frank2026}.

\begin{lemma}\label{lem:geometric-summation}
Let $p\in[1,\infty)$, $J\in\mathbb N$, and
$\varepsilon_0,\ldots,\varepsilon_{J-1}\in\{0,1\}$.
Then the following estimates hold.
\begin{enumerate}[\rm(i)]
\item For any $q\in(1,\infty)$,
$$
\left(\sum_{j=0}^{J-1}q^j\varepsilon_j\right)^p
\leq\frac{q^p}{(q-1)^p}\sum_{j=0}^{J-1}q^{jp}\varepsilon_j.
$$
\item For any $q\in(0,1)$,
\begin{equation}\label{eq:decreasing-geometric-sum}
\left(\sum_{j=0}^{J-1}q^j\varepsilon_j\right)^p
\leq\frac{1}{(1-q)^p}\sum_{j=0}^{J-1}q^{jp}\varepsilon_j.
\end{equation}
\end{enumerate}
Both estimates also hold for infinite sums over $j\geq0$ and any
sequence $\{\varepsilon_j\}_{j\geq0}\subset\{0,1\}$.
\end{lemma}

\begin{proof}
If $\varepsilon_0=\cdots=\varepsilon_{J-1}=0$, both estimates are
immediate. Otherwise, for (i), let $m$ be the largest index such that
$\varepsilon_m=1$. An elementary calculation gives
$$
\left(\sum_{j=0}^{J-1}q^j\varepsilon_j\right)^p
\leq\left(\frac{q^{m+1}-1}{q-1}\right)^p
\leq\frac{q^{p(m+1)}}{(q-1)^p}
\leq\frac{q^p}{(q-1)^p}\sum_{j=0}^{J-1}q^{jp}\varepsilon_j.
$$
For (ii), let $m$ instead be the smallest index such that
$\varepsilon_m=1$. Then
$$
\left(\sum_{j=0}^{J-1}q^j\varepsilon_j\right)^p
\leq\left(\sum_{j=m}^{\infty}q^j\right)^p
=\frac{q^{mp}}{(1-q)^p}
\leq\frac{1}{(1-q)^p}\sum_{j=0}^{J-1}q^{jp}\varepsilon_j,
$$
which gives \eqref{eq:decreasing-geometric-sum}.
Letting $J\to\infty$ and applying the monotone convergence theorem,
we obtain both assertions for infinite sums. This completes the proof
of Lemma \ref{lem:geometric-summation}.
\end{proof}

For any integer $j\in[0,\infty)$, any
$k\in\{0,\ldots,2^j-1\}$, and any $f\in L^1([0,1))$, define
$$
(\mathbb E_j f)(t)
:=2^j\int_{k2^{-j}}^{(k+1)2^{-j}}f(s)\,ds,
\quad \forall\,t\in[k2^{-j},(k+1)2^{-j}).
$$
The next lemma records the iteration for the two complementary
truncations; its first part borrows some ideas from
\cite[Lemma~2, Step~2]{Frank2026}.

\begin{lemma}\label{lem:dyadic-truncation}
Let $J\in\mathbb N$, and let $G_0,\ldots,G_J$ be real-valued functions
on $[0,1)$ such that, for any $j\in\{0,\ldots,J\}$, $G_j$ is constant on any interval
$[k2^{-j},(k+1)2^{-j})$, $k\in\{0,\ldots,2^j-1\}$, and
$G_j=\mathbb E_jG_{j+1}$ for $j\in\{0,\ldots,J-1\}$.
For $a_0,\ldots,a_J\in(0,\infty)$ and $j\in\{0,\ldots,J\}$, let
$b_j:=\mathbf1_{\{|G_j|>a_j\}}$.
Then the following estimates hold.
\begin{enumerate}[\rm(i)]
\item If $a_0\leq\cdots\leq a_J$, then
$$
(|G_0|-a_0)_+
\leq\int_0^1(|G_J(t)|-a_J)_+\,dt
+\sum_{j=0}^{J-1}(a_{j+1}-a_j)\int_0^1b_j(t)\,dt.
$$
\item If $a_0\geq\cdots\geq a_J$, then
$$
\min\{|G_0|,a_0\}
\leq a_J+4\sum_{j=0}^{J-1}a_j\int_0^1b_{j+1}(t)\,dt.
$$
\end{enumerate}
\end{lemma}

\begin{proof}
For (i), using $G_j=\mathbb E_jG_{j+1}$, we obtain
$$
|G_j|\leq\mathbb E_j|G_{j+1}|
\leq\mathbb E_j(|G_{j+1}|-a_{j+1})_++a_{j+1}.
$$
On the set where $b_j=1$, we subtract $a_j$ from this estimate.
On its complement, $(|G_j|-a_j)_+=0$. Thus,
\begin{equation*}
(|G_j|-a_j)_+
\leq\mathbb E_j(|G_{j+1}|-a_{j+1})_+
+(a_{j+1}-a_j)b_j.
\end{equation*}
Since $\int_0^1(\mathbb E_j f)(t)\,dt=\int_0^1 f(t)\,dt$ for any
$f\in L^1([0,1))$, we infer that
$$
\int_0^1(|G_j(t)|-a_j)_+\,dt
\leq\int_0^1(|G_{j+1}(t)|-a_{j+1})_+\,dt
+(a_{j+1}-a_j)\int_0^1b_j(t)\,dt.
$$
Summing over $j\in\{0,\ldots,J-1\}$ and cancelling the intermediate terms
give (i).

For (ii), let $T_a(s):=\max\{-a,\min\{s,a\}\}$ for
$a\in(0,\infty)$ and $s\in\mathbb R$, and let $V_j:=T_{a_j}(G_j)$.
We claim that
\begin{equation}\label{eq:positive-clipping-recursion}
|V_j-\mathbb E_jV_{j+1}|
\leq4a_j\mathbb E_jb_{j+1}.
\end{equation}
Indeed, on any interval of the $j$th dyadic partition, if $b_{j+1}=0$
on both children, then $|G_{j+1}|\leq a_{j+1}\leq a_j$
there, so that $V_j=G_j=\mathbb E_jG_{j+1}=\mathbb E_jV_{j+1}$.
Otherwise, $\mathbb E_jb_{j+1}\geq\frac12$, while
$|V_j-\mathbb E_jV_{j+1}|\leq a_j+a_{j+1}\leq2a_j$.
These observations prove \eqref{eq:positive-clipping-recursion}.
Using the identity
$$
V_0=\int_0^1V_J(t)\,dt+
\sum_{j=0}^{J-1}\int_0^1
\left[V_j(t)-\mathbb E_jV_{j+1}(t)\right]\,dt
$$
together with $|V_0|=\min\{|G_0|,a_0\}$, $|V_J|\leq a_J$, and
\eqref{eq:positive-clipping-recursion}, we obtain (ii).
This completes the proof of Lemma \ref{lem:dyadic-truncation}.
\end{proof}

\begin{lemma}\label{lem:dyadic-integration}
Let $N\in\mathbb N$, $p\in[1,\infty)$, and $\gamma\in\mathbb R$.
For any nonnegative measurable function $M$ on $(-1,1)^N$,
$$
\sum_{j=0}^{\infty}2^{-\gamma j}
\int_{(-1,1)^N}|h|^{p+\gamma}M(2^{-j}h)\,dh
\lesssim\int_{(-1,1)^N}|z|^{\gamma-N}M(z)\,dz,
$$
where the implicit positive constant depends only on $N$ and $p$.
\end{lemma}

\begin{proof}
Fix any $z\in(-1,1)^N\setminus\{\mathbf{0}\}$ and choose the integer
$m\geq0$ such that
$$
2^m\leq\frac{\sqrt N}{|z|}<2^{m+1}.
$$
The condition $z\in(-2^{-j},2^{-j})^N$ implies
$|z|\leq\sqrt N\,2^{-j}$, and hence $j\leq m$.
An elementary calculation gives
\begin{align}
\sum_{j=0}^{\infty}2^{(N+p)j}\mathbf{1}_{(-2^{-j},2^{-j})^N}(z)
&\leq\sum_{j=0}^{m}2^{(N+p)j}
=\frac{2^{(N+p)(m+1)}-1}{2^{N+p}-1}\notag\\
&\lesssim2^{(N+p)m}
\leq\left(\frac{\sqrt N}{|z|}\right)^{N+p}
\lesssim|z|^{-N-p}.
\label{eq:dyadic-kernel}
\end{align}
By the change of variables $z=2^{-j}h$, Tonelli's theorem, and
\eqref{eq:dyadic-kernel}, we obtain
\begin{align*}
&\sum_{j=0}^{\infty}2^{-\gamma j}
\int_{(-1,1)^N}|h|^{p+\gamma}M(2^{-j}h)\,dh\\
&\quad=\int_{(-1,1)^N}|z|^{p+\gamma}M(z)
\sum_{j=0}^{\infty}2^{(N+p)j}\mathbf{1}_{(-2^{-j},2^{-j})^N}(z)\,dz\\
&\quad\lesssim\int_{(-1,1)^N}|z|^{\gamma-N}M(z)\,dz.
\end{align*}
This completes the proof of Lemma \ref{lem:dyadic-integration}.
\end{proof}

With these preparations, we are now ready
to show Theorem \ref{thm:truncated-differences}.
\begin{proof}[Proof of Theorem \ref{thm:truncated-differences}]
By translation and dilation, we conclude that it suffices to consider
$Q=Q_0:=(0,1)^N$. Indeed, if
$Q=x_0+\ell(Q)Q_0$, then
the function $v(z):=g(x_0+\ell(Q)z)$ and the threshold
$\mu:=\lambda\ell(Q)^{1+\frac{\gamma}{p}}$ satisfy
$$
\Phi_{\mu,p}^{\gamma}(v;Q_0)
=\ell(Q)^{p-N}\Phi_{\lambda,p}^{\gamma}(g;Q).
$$
The left-hand sides transform into the corresponding integral
averages on $Q$. Applying the estimates with $\lambda=1$ to
$\frac{g}{\lambda}$ then reduces the proof to $\lambda=1$.

Assume first that $g$ is bounded and measurable. Let
$\beta:=1+\frac{\gamma}{p}$.
For $h\in(-1,1)^N$, define
$$
Q_{0,h}:=\{x\in Q_0:x+h\in Q_0\}\ \text{and}\
\Delta_hg(x):=g(x+h)-g(x)
$$
for any $x\in Q_{0,h}$.
For $h\in(-1,1)^N\setminus\{\mathbf{0}\}$, let
$$
M(h):=\left|\left\{x\in Q_{0,h}:
\frac{|\Delta_hg(x)|}{|h|^\beta}>1\right\}\right|,
$$ and let $M(\mathbf{0}):=0.$
By the change of variables $y=x+h$ and Tonelli's theorem, we obtain
\begin{equation}\label{eq:energy-translation}
\mathcal E
:=\iint\nolimits_{E_{1,\frac{\gamma}{p}}[g;Q_0]}
|x-y|^{\gamma-N}\,dx\,dy
=\int_{(-1,1)^N}|h|^{\gamma-N}M(h)\,dh.
\end{equation}
We may assume that $\mathcal E<\infty$ because otherwise the statement
is immediate.

Fix $h\in(-1,1)^N\setminus\{\mathbf{0}\}$ and $x\in Q_{0,h}$.
Since $Q_0$ is convex, it follows that the segment joining $x$ to $x+h$
is contained in $Q_0$.  Given an integer $j\ge 0$,
define a function $G_j$ on $[0,1)$ by
\begin{equation}\label{eq:dyadic-increments}
G_j(t):=2^j\Delta_{2^{-j}h}g(x+k2^{-j}h),
\quad
t\in[k2^{-j},(k+1)2^{-j}),\quad k\in\{0,\ldots,2^j-1\}.
\end{equation}
By the definition of $G_j$ and the identity
$$
\Delta_{2^{-j}h}g(\cdot)
=\Delta_{2^{-j-1}h}g(\cdot)
+\Delta_{2^{-j-1}h}g(\cdot+2^{-j-1}h)
$$
for any $j\in \mathbb Z_+$ and $h\in(-1,1)^N$,
we obtain $G_j=\mathbb E_jG_{j+1}$
and $G_0=\Delta_hg(x).$
Define
\begin{equation}\label{eq:thresholds}
a_j:=2^j|2^{-j}h|^\beta
=2^{-\frac{\gamma j}{p}}|h|^\beta
\ \text{and}\
b_j(t):=\mathbf{1}_{\{|G_j(t)|>a_j\}}.
\end{equation}
For any $k\in\{0,\ldots,2^j-1\}$, the convexity of $Q_0$ gives
$Q_{0,h}+k2^{-j}h\subset Q_{0,2^{-j}h}$.
From \eqref{eq:dyadic-increments} and \eqref{eq:thresholds}, we deduce
that
\begin{align}
\int_{Q_{0,h}}\int_0^1b_j(t)\,dt\,dx
=2^{-j}\sum_{k=0}^{2^j-1}\int_{Q_{0,h}}
\mathbf{1}_{\{\frac{|\Delta_{2^{-j}h}g(x+k2^{-j}h)|}{|2^{-j}h|^\beta}>1\}}\,dx
\leq M(2^{-j}h).
\label{eq:indicator-average}
\end{align}

Next, we consider the following two cases for $\gamma$.

\emph{Case 1: $\gamma\in(-\infty,-1]$.}
In this case, since $a_{j+1}=2^{-\frac{\gamma}{p}}a_j$ and
$2^{-\frac{\gamma}{p}}\in(1,\infty)$, we infer from
Lemma~\ref{lem:dyadic-truncation}(i) that, for any $J\in\mathbb N$,
\begin{equation}\label{eq:finite-iteration}
\left(|\Delta_hg(x)|-|h|^\beta\right)_+
\leq\int_0^1H_J(t)\,dt+R_J(x,h),
\end{equation}
where
\begin{equation*}
H_J(t):=\sum_{j=0}^{J-1}(a_{j+1}-a_j)b_j(t)\ \text{and}\
R_J(x,h):=\int_0^1(|G_J(t)|-a_J)_+\,dt.
\end{equation*}
For any $t\in(0,1)$,
applying Lemma~\ref{lem:geometric-summation}(i) with $q=2^{-\frac{\gamma}{p}}$ and
$\varepsilon_j=b_j(t)$, and using
$a_{j+1}-a_j=(q-1)q^j|h|^\beta$, we obtain
\begin{equation}\label{eq:geometric-sum}
[H_J(t)]^p
\leq2^{-\gamma}\sum_{j=0}^{J-1}a_j^pb_j(t).
\end{equation}
From \eqref{eq:geometric-sum} and \eqref{eq:indicator-average},
it follows that
\begin{align}
&\int_{(-1,1)^N}\int_{Q_{0,h}}
\left[\int_0^1H_J(t)\,dt\right]^p\,dx\,dh
\notag\\
&\quad\leq\int_{(-1,1)^N}\int_{Q_{0,h}}\int_0^1[H_J(t)]^p\,dt\,dx\,dh
\notag\\
&\quad\leq2^{-\gamma}\sum_{j=0}^{J-1}2^{-\gamma j}
\int_{(-1,1)^N}|h|^{p+\gamma}M(2^{-j}h)\,dh.
\label{eq:sum-before-kernel}
\end{align}
Applying Lemma~\ref{lem:dyadic-integration} to \eqref{eq:sum-before-kernel} and using
\eqref{eq:energy-translation}, we find that
\begin{equation}\label{eq:uniform-main-sum}
\int_{(-1,1)^N}\int_{Q_{0,h}}
\left[\int_0^1H_J(t)\,dt\right]^p\,dx\,dh
\lesssim2^{|\gamma|}\mathcal E,
\end{equation}
where the implicit positive constant depends only on $N$ and $p$.

It remains to estimate the remainder in \eqref{eq:finite-iteration}.
Since $g$ is bounded, we deduce that, for any $t\in(0,1)$,
$$
(|G_J(t)|-a_J)_+
\leq2^{J+1}\lVert g\rVert _{L^\infty(Q_0)}b_J(t).
$$
By \eqref{eq:indicator-average} and the change of variables
$z=2^{-J}h$, we obtain
\begin{align*}
\int_{(-1,1)^N}\int_{Q_{0,h}}R_J(x,h)\,dx\,dh
&\leq2^{J+1}\lVert g\rVert _{L^\infty(Q_0)}
\int_{(-1,1)^N} M(2^{-J}h)\,dh
\notag\\
&=2\lVert g\rVert _{L^\infty(Q_0)}2^{(N+1)J}
\int_{(-2^{-J},2^{-J})^N}M(z)\,dz.
\end{align*}
For $z\in(-2^{-J},2^{-J})^N\setminus\{\mathbf{0}\}$, we have
$$
2^{(N+1)J}
=2^{(1+\gamma)J}|2^Jz|^{N-\gamma}|z|^{\gamma-N}
\leq N^{\frac{N-\gamma}{2}}2^{(1+\gamma)J}|z|^{\gamma-N}.
$$
Consequently, as $J\to\infty$,
\begin{equation}\label{eq:remainder-vanishes}
\int_{(-1,1)^N}\int_{Q_{0,h}}R_J(x,h)\,dx\,dh
\lesssim\lVert g\rVert _{L^\infty(Q_0)}2^{(1+\gamma)J}
\int_{(-2^{-J},2^{-J})^N}|z|^{\gamma-N}M(z)\,dz
\to 0,
\end{equation}
where the implicit positive constant is independent of $J$.
Indeed, $2^{(1+\gamma)J}\leq1$ because $\gamma\leq-1$, and the
integral on the right tends to zero by
\eqref{eq:energy-translation} and $\mathcal E<\infty$.

For any $a,r\in[0,\infty)$, observe that
$$
a^p-(a-r)_+^p
=p\int_{(a-r)_+}^{a}t^{p-1}\,dt
\leq p(1+a)^{p-1}r.
$$
Applying this inequality and \eqref{eq:finite-iteration}, and using
$(|\Delta_hg(x)|-|h|^\beta)_+\leq2\lVert g\rVert _{L^\infty(Q_0)}$,
we obtain
$$
\left(|\Delta_hg(x)|-|h|^\beta\right)_+^p
\leq\left[\int_0^1H_J(t)\,dt\right]^p
+p\left(1+2\lVert g\rVert _{L^\infty(Q_0)}\right)^{p-1}R_J(x,h).
$$
Integrating this inequality, using \eqref{eq:uniform-main-sum} and
\eqref{eq:remainder-vanishes}, and letting $J\to\infty$, we conclude that
\begin{equation}\label{eq:truncated-increment-bound}
\int_{(-1,1)^N}\int_{Q_{0,h}}
\left(|\Delta_hg(x)|-|h|^\beta\right)_+^p\,dx\,dh
\lesssim2^{|\gamma|}\mathcal E,
\end{equation}
where the implicit positive constant depends only on $N$ and $p$.
By the change of variables $y=x+h$, we obtain
\eqref{eq:truncation-negative} for $Q=Q_0$ and $\lambda=1$.

\emph{Case 2: $\gamma\in(0,\infty)$.}
In this case, from $a_{j+1}\leq a_j$, the convergence $a_J\to0$ as $J\to\infty$, and
Lemma~\ref{lem:dyadic-truncation}(ii), we infer that
\begin{equation*}
\min\{|\Delta_hg(x)|,a_0\}
\leq a_J+4\int_0^1\sum_{j=0}^{J-1}a_jb_{j+1}(t)\,dt.
\end{equation*}
Letting $J\to\infty$,
we obtain
\begin{equation}\label{eq:positive-infinite-iteration}
\min\{|\Delta_hg(x)|,|h|^\beta\}
\leq4\int_0^1\sum_{j=0}^{\infty}a_jb_{j+1}(t)\,dt.
\end{equation}
Applying Lemma~\ref{lem:geometric-summation}(ii) with
$q=2^{-\frac{\gamma}{p}}$ and $\varepsilon_j=b_{j+1}(t)$, and using
\eqref{eq:positive-infinite-iteration} and \eqref{eq:indicator-average},
we conclude that
\begin{align*}
&\int_{(-1,1)^N}\int_{Q_{0,h}}
\min\{|\Delta_hg(x)|^p,|h|^{p+\gamma}\}\,dx\,dh\\
&\quad\leq\frac{4^p}{(1-2^{-\frac{\gamma}{p}})^p}
\sum_{j=0}^{\infty}2^{-\gamma j}
\int_{(-1,1)^N}|h|^{p+\gamma}M(2^{-j-1}h)\,dh\\
&\quad=\frac{4^p2^\gamma}{(1-2^{-\frac{\gamma}{p}})^p}
\sum_{j=1}^{\infty}2^{-\gamma j}
\int_{(-1,1)^N}|h|^{p+\gamma}M(2^{-j}h)\,dh
\lesssim\frac{4^p2^\gamma}{(1-2^{-\frac{\gamma}{p}})^p}
\mathcal E,
\end{align*}
where the last step used Lemma~\ref{lem:dyadic-integration} and
\eqref{eq:energy-translation}, and where the implicit positive constant
depends only on $N$ and $p$.
By the change of variables $y=x+h$, we obtain
\eqref{eq:truncation-positive} for $Q=Q_0$ and $\lambda=1$.

Finally, let $g:Q_0\to\mathbb R$ be an arbitrary measurable function.
For $m\in\mathbb N$ and $x\in Q_0$, let
$$
g_m(x):=\max\{-m,\min\{g(x),m\}\}.
$$
Since $|g_m(x)-g_m(y)|\leq|g(x)-g(y)|$, it follows that
$$
E_{1,\frac{\gamma}{p}}[g_m;Q_0]
\subset E_{1,\frac{\gamma}{p}}[g;Q_0].
$$
Applying the estimates proved above to $g_m$ and then using
$g_m\to g$ pointwise as $m\to\infty$ and Fatou's lemma, we obtain
the same estimates for $g$.
This completes the proof
of Theorem~\ref{thm:truncated-differences}.
\end{proof}

\begin{remark}
The proof gives the factors $2^{|\gamma|}$ in
\eqref{eq:truncation-negative} and
$\frac{4^p2^\gamma}{(1-2^{-\frac{\gamma}{p}})^p}$ in
\eqref{eq:truncation-positive}, with implicit positive constants
depending only on $N$ and $p$.
In particular, the constant in \eqref{eq:truncation-negative} is
bounded uniformly when $\gamma$ ranges over a compact subinterval of
$(-\infty,-1]$.
\end{remark}

\section{Sharp Nonlocal Poincar\'e Inequalities}
\label{sec:main-theorem-completion}

The main goal of this section is to explain how we deduce
the Poincar\'e inequality \eqref{1-1} from
Theorem~\ref{thm:truncated-differences} in Subsection \ref{s3.1}.
We then show the failure outside
$\gamma\in(-\infty,-1]$ in Subsection \ref{sec:failure},
study the dependence of the constant
in \eqref{1-1} on
$\gamma$ in Subsection \ref{s3.3},
and deduce the positivity of the cell constant in Subsection \ref{s3.4}.
These results complete the proof of Corollary~\ref{thm:classification}.
Finally, we illustrates how the Poincar\'e inequality
\eqref{1-1} recovers the classical one in Subsection \ref{s3.5}.

By translation and dilation, as in the proof of
Theorem~\ref{thm:truncated-differences}, we conclude that it
suffices to prove \eqref{1-1} and the assertions concerning its
validity and least constant on the unit cube $Q_0:=(0,1)^N$.
We work on $Q_0$ in the following proofs.

\subsection{The Poincar\'e Inequality for
\texorpdfstring{$\gamma\in(-\infty,-1]$}{gamma in (-infinity, -1]}}\label{s3.1}

In this subsection, we use \eqref{eq:truncation-negative} to obtain
the Poincar\'e inequality \eqref{1-1} for
$\gamma\in(-\infty,-1]$.

\begin{corollary}\label{thm:unified-negative}
Let $N\in\mathbb N$, $p\in[1,\infty)$, and $\gamma\in(-\infty,-1]$.
Then there exists a constant $c_{N,p}\in(0,\infty)$, depending only on
$N$ and $p$, such that, for any $\lambda\in(0,\infty)$ and
any measurable function $g:Q_0\to\mathbb R$,
\begin{equation*}
\int_{Q_0}\int_{Q_0}|g(x)-g(y)|^p\,dx\,dy
\leq c_{N,p}8^{|\gamma|}\lambda^p
\left(1+
\iint\nolimits_{E_{\lambda,\frac{\gamma}{p}}[g;Q_0]}
|x-y|^{\gamma-N}\,dx\,dy\right).
\end{equation*}
\end{corollary}

\begin{proof}
Applying the desired estimate in the present corollary
with $\lambda=1$ to $\frac{g}{\lambda}$,
we conclude that it suffices to consider $\lambda=1$.
Let $\beta:=1+\frac{\gamma}{p}$.
We retain the notation $Q_{0,h}$, $\Delta_hg$, and $\mathcal E$
from the proof of Theorem~\ref{thm:truncated-differences}.
Let
$$
\mathcal A:=\left\{h\in(-1,1)^N:|h|\geq\frac{1}{8}\right\}.
$$
For $x,y\in Q_0$, define
$$
Z_{x,y}:=Q_0\setminus
\left[B\left(x,\frac{1}{8}\right)\cup B\left(y,\frac{1}{8}\right)\right].
$$
Since a ball of radius $\frac{1}{8}$ is contained in a cube of edge length
$\frac{1}{4}$, we infer that
$$
|Z_{x,y}|\geq1-2\cdot4^{-N}\geq\frac{1}{2}.
$$
For any $z\in Z_{x,y}$, we have
$$
|g(x)-g(y)|^p
\leq2^{p-1}\left[|g(x)-g(z)|^p+|g(z)-g(y)|^p\right].
$$
Averaging over $z\in Z_{x,y}$ and then integrating in $x,y\in Q_0$
give
\begin{align}
\int_{Q_0}\int_{Q_0}|g(x)-g(y)|^p\,dx\,dy
&\leq2^{p+1}
\iint_{\genfrac{}{}{0pt}{}{x,z\in Q_0}{|x-z|\geq\frac{1}{8}}}
|g(x)-g(z)|^p\,dx\,dz
\notag\\
&=2^{p+1}\int_{\mathcal A}\int_{Q_{0,h}}
|\Delta_hg(x)|^p\,dx\,dh.
\label{eq:far-increments}
\end{align}
Note that
$$
|\Delta_hg(x)|^p
\leq2^{p-1}\left[
\left(|\Delta_hg(x)|-|h|^\beta\right)_+^p
+|h|^{p+\gamma}\right].
$$
Moreover, since $\frac{1}{8}\leq|h|\leq\sqrt N$ on $\mathcal A$ and
since $\gamma\leq-1$, it follows that
$$
|h|^{p+\gamma}=|h|^p|h|^\gamma
\leq N^{\frac{p}{2}}8^{|\gamma|},\quad \forall\,h\in\mathcal A.
$$
Using these estimates together with \eqref{eq:far-increments} and
\eqref{eq:truncated-increment-bound}, we obtain
\begin{align*}
\int_{Q_0}\int_{Q_0}|g(x)-g(y)|^p\,dx\,dy
&\lesssim2^{|\gamma|}\mathcal E
+\int_{\mathcal A}\int_{Q_{0,h}}|h|^{p+\gamma}\,dx\,dh\\
&\lesssim2^{|\gamma|}\mathcal E+8^{|\gamma|}
\leq8^{|\gamma|}(1+\mathcal E),
\end{align*}
where the implicit positive constants depend only on $N$ and $p$.
This completes the proof of Corollary \ref{thm:unified-negative}.
\end{proof}

\begin{remark}\label{rem:elementary-nguyen-proof}
Brezis and Nguyen \cite[Remark~2]{BrezisNguyenVMO} noted that
the proof of \cite[Proposition~1]{BrezisNguyenVMO} relies on the
technique developed by Bourgain and Nguyen \cite{BourgainNguyen2006} and
the John--Nirenberg inequality, and then asked for a more elementary argument.
For $\gamma=-p$, Corollary~\ref{thm:unified-negative} reduces to
\cite[Proposition~1]{BrezisNguyenVMO} with an elementary proof based on dyadic iteration for
differences along line segments, without using the John--Nirenberg
inequality. This answers the aforementioned question of Brezis
and Nguyen.
\end{remark}

\subsection{\texorpdfstring{Failure of the Poincar\'e Inequality for $\gamma\in(-1,\infty)$}{Failure of the inequality for gamma greater than -1}}
\label{sec:failure}

This subsection is devoted to showing that \eqref{1-1} fails for $\gamma\in(-1,\infty)$.
We treat $\gamma\in(0,\infty)$ by letting $\lambda\to0^+$ and
$\gamma\in(-1,0]$ by constructing a counterexample.

\begin{proposition}
\label{prop:positive-gamma-global}
Let $\gamma\in(0,\infty)$ and $g\in L^p(Q_0)$.  Then no finite constant can make
\eqref{1-1} valid for any $\lambda\in(0,\infty)$ unless $g$ is  constant almost everywhere on $Q_0$.
\end{proposition}

\begin{proof}
If $\gamma\in(0,\infty)$, then the function
$(x, y)\mapsto \lvert x-y\rvert ^{\gamma-N}$ is integrable on
$Q_0\times Q_0$, and hence,
for any $\lambda\in(0,\infty)$,
$$\lambda^p \iint\nolimits_{E_{\lambda,\frac{\gamma}{p}}[g;Q_0]}\lvert x-y\rvert ^{\gamma-N}\,dx\,dy\leq \lambda^p
 \iint_{Q_0\times Q_0}\lvert x-y\rvert ^{\gamma-N}\,dx\,dy\lesssim\lambda^p.
$$
If \eqref{1-1} holds, then we obtain
  $$ \int_{Q_0}\int_{Q_0} |g(x)-g(y)|^p\,dx\,dy \lesssim\lambda^p\to 0
  \quad\text{as }\lambda\to 0^+,$$
where the implicit positive constant is independent of $\lambda$. Consequently,
  $$  \int_{Q_0}\int_{Q_0} |g(x)-g(y)|^p\,dx\,dy=0.$$
  But this  is impossible unless $g$ is constant almost everywhere on $Q_0$.
This completes the proof of Proposition \ref{prop:positive-gamma-global}.
\end{proof}

\begin{remark}
The obstruction in the proof of Proposition \ref{prop:positive-gamma-global}
arises as $\lambda\to 0^+$, suggesting that an asymptotic
substitute for \eqref{1-1} may hold for sufficiently large $\lambda$.
Indeed,
Theorem~\ref{thm:truncated-differences}(ii) gives a Poincar\'e estimate
for sufficiently large thresholds when $\gamma\in(0,\infty)$.
More precisely, for any $g\in L^p(Q_0)$, there exists
$\lambda_0=\lambda_0(g,N,p,\gamma)\in(0,\infty)$ such that,
for any $\lambda\in[\lambda_0,\infty),$
\begin{equation}\label{eq:positive-gamma-large-lambda}
\int_{Q_0}\int_{Q_0}|g(x)-g(y)|^p\,dx\,dy
\lesssim\Phi_{\lambda,p}^\gamma(g;Q_0),
\end{equation}
where the implicit positive constant depends only on $N$, $p$, and $\gamma$.
To prove it, by the monotone convergence theorem, the left-hand side of
\eqref{eq:truncation-positive} with $Q=Q_0$ increases to
$$\int_{Q_0}\int_{Q_0}|g(x)-g(y)|^p\,dx\,dy$$
as $\lambda\to\infty$. If $g$ is not constant almost everywhere,
this limit is positive and
finite, so the truncated integral is at least half of its limit for
sufficiently large $\lambda$. Applying \eqref{eq:truncation-positive},
we obtain \eqref{eq:positive-gamma-large-lambda}.
\end{remark}

\begin{proposition}
\label{prop:minus-one-sharp-global}
If $\gamma\in(-1,0]$ and $p\in[1,\infty)$,  then no finite constant can make
\eqref{1-1} valid for any $\lambda\in(0,\infty)$ and any measurable $g: Q_0\to\mathbb R$.
\end{proposition}

\begin{proof}
Let $Q':=(\frac{1}{4},\frac{3}{4})^N\subset Q_0$ and
 $g:=\mathbf{1}_{Q'}$.
 Assume otherwise that \eqref{1-1} holds for some $\gamma\in(-1,0]$.
Clearly,
\begin{align*}
 \int_{Q_0}\int_{Q_0} |g(x)-g(y)|^p\,dx\,dy &=2|Q'\times (Q_0\setminus Q')|\\
&=2^{-N+1}(1-2^{-N})\ge 2^{-N}.
\end{align*}
Since
$$
E_{\lambda,\frac{\gamma}{p}}[g;Q_0]
\subset\left[Q'\times(Q_0\setminus Q')\right]
\cup\left[(Q_0\setminus Q')\times Q'\right],
$$
we deduce that
\begin{align*}\lambda^p\iint\nolimits_{E_{\lambda,\frac{\gamma}{p}}[g;Q_0]} |x-y|^{\gamma-N}\,dx\,dy \leq 2\lambda^p\int_{Q'}\int_{Q_0\setminus Q'}
 \lvert x-y\rvert ^{\gamma-N}\,dy\,dx.
\end{align*}
We claim this double integral in the right-hand side
 is finite. To see this, for any $r\in(0,N]$, let
$$ G_r:=\left\{(x,y):\ x\in Q',\ y\in Q_0\setminus Q', \ |x-y|\leq r\right\}.$$
Since
$$ G_r\subset \left\{ (x, y):\ x\in Q',\ \min_{z\in\partial Q'} |x-z|\leq  r,\ |x-y|\leq r\right\},$$
we infer that $|G_r|\lesssim r^{N+1}$ for any $r\in(0,N]$. Thus,
\begin{align*}\int_{Q'}\int_{Q_0\setminus Q'}
 \lvert x-y\rvert ^{\gamma-N}\,dy\,dx&\leq \sum_{j=0}^\infty \iint_{G_{2^{-j} N}\setminus G_{2^{-j-1}N}} \lvert x-y\rvert ^{\gamma-N}\,dy\,dx\\
 &\lesssim \sum_{j=0}^\infty 2^{-j(\gamma-N)} |G_{2^{-j} N}|
 \lesssim \sum_{j=0}^\infty 2^{-j(\gamma+1)}\lesssim 1.\end{align*}

Using these estimates and \eqref{1-1}, we find that, for any $\lambda\in(0,\infty)$,
$2^{-N}\lesssim\lambda^p,$
where the implicit positive constant is independent of $\lambda$.
Letting $\lambda\to 0^+$ gives a contradiction.
This completes the proof of Proposition \ref{prop:minus-one-sharp-global}.
\end{proof}

Using similar examples, we can show that
the ranges in Theorem \ref{thm:truncated-differences} are sharp, as stated in
Remark \ref{rem:truncation-sharpness}.
\begin{proof}[Proof of Remark~\ref{rem:truncation-sharpness}]
Let $g=\mathbf1_{Q'}$, where $Q'=(\frac{1}{4},\frac{3}{4})^N$.
By the estimates in the proofs of Propositions
\ref{prop:positive-gamma-global} and \ref{prop:minus-one-sharp-global},
we conclude that, for any $\gamma\in(-1,\infty)$,
$\Phi_{\lambda,p}^{\gamma}(g;Q_0)\lesssim\lambda^p\to0$
as $\lambda\to0^+,$
where the implicit positive constant is independent of $\lambda$.
The dominated convergence theorem gives
$$
\lim_{\lambda\to0^+}\int_{Q_0}\int_{Q_0}
\left(|g(x)-g(y)|-\lambda|x-y|^{1+\frac{\gamma}{p}}\right)_+^p
\,dx\,dy
=2|Q'|(1-|Q'|)>0,
$$
which rules out \eqref{eq:truncation-negative}.

For $\gamma\in(-\infty,0]$, let $g(x):=x_1$ for $x:=(x_1,\ldots,x_N)\in\mathbb{R}^N$,
and choose
$\lambda\in[(\sqrt N)^{-\frac{\gamma}{p}},\infty)$.
Since $|x-y|\leq\sqrt N$ for $x,y\in Q_0$, it follows that
$$
|g(x)-g(y)|\leq|x-y|
\leq\lambda|x-y|^{1+\frac{\gamma}{p}},\quad \forall\,x\ne y.
$$
Thus, $\Phi_{\lambda,p}^{\gamma}(g;Q_0)=0$, whereas the left-hand
side of \eqref{eq:truncation-positive} equals
$$
\int_{Q_0}\int_{Q_0}|x_1-y_1|^p\,dx\,dy
=\frac{2}{(p+1)(p+2)}>0.
$$
This rules out \eqref{eq:truncation-positive} and completes the proof.
\end{proof}

\subsection{\texorpdfstring{Dependence of the
Constant in \eqref{1-1} on $\gamma$}{Dependence of the constant on gamma}}\label{s3.3}

In this subsection, we use affine functions to prove that the least
constant in \eqref{1-1} grows at least linearly in $|\gamma|$
as $\gamma\to-\infty$.

\begin{proposition}\label{prop:constant-negative-infinity} Assume that $p\in[1,\infty)$ and  $\gamma\in(-\infty,0)$.
Let $c_{N,p,\gamma}$ denote the
least constant in  \eqref{1-1}. Then
\begin{equation}\label{5-1}
 c_{N,p,\gamma}\geq
 \frac{2\lvert \gamma\rvert }{(p+1)(p+2)K_{N,p}},
\end{equation}
where $K_{N,p}$ is as in \eqref{1.4x}.
\end{proposition}

\begin{proof}
We first recall the following formula for any
$\gamma\in\mathbb R\setminus\{0\}$ and $e\in\mathbb S^{N-1}$.
Let $\beta:=1+\frac{\gamma}{p}$ and $Q_0:=(0,1)^N$.
We write $(x,y)$ for the Euclidean inner product.
For $\lambda\in(0,\infty)$, define
$$
A_{\lambda,e}:=\left\{(x,y)\in Q_0\times Q_0:x\ne y,\
\frac{|(x,e)-(y,e)|}{|x-y|^\beta}>\lambda\right\}.
$$
Then
\begin{equation}\label{aff:eq:affine-limit}
 \lim_{\lambda\to L_\gamma}
 \lambda^p\iint\nolimits_{A_{\lambda,e}}|x-y|^{\gamma-N}\,dx\,dy
 =\frac{K_{N,p}}{|\gamma|}.
\end{equation}
From
\cite[Theorem~9.1(i)]{DominguezMilman}, with polar coordinates
$y=x+t\omega$, $X=Q_0\times\mathbb S^{N-1}$ [equipped with measure
$dm(x,\omega)=dx\,d\sigma(\omega)$],
$T_tF(x,\omega):=\mathbf{1}_{Q_0}(x+t\omega)F(x,\omega)$ for any $t\in(0,\infty)$,
and $f(x,\omega):=|\langle e,\omega\rangle|$, we deduce that
\eqref{aff:eq:affine-limit} holds.
Indeed, $\sup_{t\in(0,\infty)}|T_tf|=f\in L^p(X,m)$ and
$T_tf\to f$ pointwise as $t\to0^+$, so the hypotheses of that theorem
are satisfied.

Returning to $\gamma\in(-\infty,0)$, we specify $e=(1,0,\ldots, 0)$ and define  $g(x):=\langle x, e\rangle=x_1$ for $x:=(x_1,\ldots,x_N)\in \mathbb R^N$.
Then
$$ \int_{Q_0}\int_{Q_0}|g(x)-g(y)|^p \,dx\,dy =\int_0^1\int_0^1 |s-t|^p\,ds\,dt =\frac{2}{(p+1)(p+2)}.$$
Applying \eqref{1-1} to $g$ and using \eqref{aff:eq:affine-limit}, we obtain
\begin{align*}\frac{2}{(p+1)(p+2)}\leq c_{N,p,\gamma}  \lim_{\lambda\to 0^+}
 \lambda^p\iint\nolimits_{A_{\lambda,e}}|x-y|^{\gamma-N}\,dx\,dy=c_{N,p,\gamma}  \frac{K_{N,p}}{|\gamma|},
\end{align*}
implying the desired inequality \eqref{5-1}.
This completes the proof of Proposition \ref{prop:constant-negative-infinity}.
\end{proof}

\subsection{Positivity of Cell Constants \texorpdfstring{$\kappa_{N,p,\gamma}$}{kappa(N,p,gamma)}}\label{s3.4}

We aim to apply the complementary truncation estimates
to establish the positivity of the cell constant for
$\gamma\in\mathbb R\setminus(-1,0]$ in this subsection.

Let $h_0(x):=\frac{x_1+\cdots+x_N}{\sqrt N}$ for $x:=(x_1,\ldots,x_N)\in Q_0$.
Recall that the \emph{cell constant} is defined by
\begin{equation}\label{6-1-0}
\kappa_{N,p,\gamma}
:=\inf_{\{h_\lambda\}}
\left[\varliminf_{\lambda\to L_\gamma}
\Phi_{\lambda,p}^{\gamma}(h_\lambda;Q_0)\right],
\end{equation}
where the infimum is taken over all families of measurable functions
$h_\lambda:Q_0\to\mathbb R$ converging to $h_0$ in measure on $Q_0$.

\begin{corollary}\label{cor:kappa-all-negative}
Let $N\in\mathbb N$ and $p\in[1,\infty)$.
For any $\gamma\in\mathbb R\setminus(-1,0]$,
$$
0<\kappa_{N,p,\gamma}\leq\frac{K_{N,p}}{|\gamma|},
$$
where $K_{N,p}$ is as in \eqref{1.4x}.
\end{corollary}

\begin{remark}
We compare Corollary~\ref{cor:kappa-all-negative} with the results on the cell constant recorded in \cite[Section~8]{NguyenSurvey}.
For $\gamma=-p$, Nguyen \cite{Nguyen2007,Nguyen2010} showed that
$0<\kappa_{N,p,-p}<\frac{K_{N,p}}{p}$,
and identified the $\Gamma$-limit in $L^p(\mathbb R^N)$ of the corresponding nonlocal functionals as $\kappa_{N,p,-p}\mathcal E_p(\cdot;\mathbb R^N)$.
The strict upper bound of $\kappa_{N,p,-p}$ exhibits a gap between the $\Gamma$-limit
and the pointwise limit on smooth functions.
Moreover, the exact value conjectured by Nguyen \cite{Nguyen2007} in
dimension one was determined in any dimension by Antonucci et al.\ \cite[Theorem~1.1 and (1-6)]{AntonucciEtAl}; in our notation,
$$
\kappa_{N,p,-p}=
\begin{cases}
\displaystyle\frac{K_{N,p}}{p(p-1)}\left(1-2^{1-p}\right)&\text{if }
p\in(1,\infty),\\
K_{N,1}\log 2&\text{if }p=1.
\end{cases}
$$

For $\gamma\in(-1,0)$, we infer from the
$\Gamma$-limsup statement in \cite[Theorem 29]{NguyenSurvey}
that $\kappa_{N,p,\gamma}=0$.
From \cite[Theorem~1.1]{Picenni2026}, it follows that the cell constant
is positive for $\gamma\in(-\infty,-p)$, while the case for $\gamma\in(0,\infty)$ is
contained in \cite{GobbinoPicenni,LinYangYangYuanZhang}.
These, together with
Corollary~\ref{cor:kappa-all-negative}, further imply that,
for any $p\in[1,\infty)$,
$$
\kappa_{N,p,\gamma}\in(0,\infty)
\quad\Longleftrightarrow\quad
\gamma\in\mathbb R\setminus(-1,0],
$$
giving a sharp characterization in terms of
$\gamma$ for the positivity of $\kappa_{N,p,\gamma}$
with $p\in[1,\infty)$.
\end{remark}

We now prove Corollary \ref{cor:kappa-all-negative}.
\begin{proof}[Proof of Corollary \ref{cor:kappa-all-negative}]
Let $h_\lambda\to h_0$ in measure on $Q_0$ as $\lambda\to L_\gamma$.
Choose a sequence $\{\lambda_n\}_{n\in\mathbb N}\subset(0,\infty)$
such that $\lambda_n\to L_\gamma$ and
\begin{equation*}
\varliminf_{\lambda\to L_\gamma}\Phi_{\lambda,p}^{\gamma}(h_\lambda;Q_0)
=\lim_{n\to\infty}
\Phi_{\lambda_n,p}^{\gamma}(h_{\lambda_n};Q_0).
\end{equation*}
Let $m\in(\sqrt N,\infty)$ and
$T_m(s):=\max\{-m,\min\{s,m\}\}$ for $s\in\mathbb R$.
Since $T_m(h_0)=h_0$, we deduce that
$T_m(h_{\lambda_n})\to h_0$ in measure on $Q_0$ as $n\to\infty$.
Moreover, the contraction property of $T_m$ gives
$$
\Phi_{\lambda_n,p}^{\gamma}(T_m(h_{\lambda_n});Q_0)
\leq\Phi_{\lambda_n,p}^{\gamma}(h_{\lambda_n};Q_0).
$$
Passing to a further subsequence, we may assume that
$T_m(h_{\lambda_n})\to h_0$ almost everywhere on $Q_0$.
For $x\ne y$, the threshold
$\lambda_n|x-y|^{1+\frac{\gamma}{p}}$ tends to zero when
$\gamma\in(-\infty,-1]$ and to infinity when $\gamma\in(0,\infty)$.
Thus, the integrands on the left-hand sides of \eqref{eq:truncation-negative}
and \eqref{eq:truncation-positive} in
Theorem~\ref{thm:truncated-differences}, applied to $T_m(h_{\lambda_n})$,
converge almost everywhere on $Q_0\times Q_0$ to $|h_0(x)-h_0(y)|^p$.
Applying Theorem~\ref{thm:truncated-differences} and Fatou's lemma, we obtain
\begin{align*}
0&<\int_{Q_0}\int_{Q_0}|h_0(x)-h_0(y)|^p\,dx\,dy\notag\\
&\lesssim\varliminf_{n\to\infty}
\Phi_{\lambda_n,p}^{\gamma}(T_m(h_{\lambda_n});Q_0)
\leq\varliminf_{\lambda\to L_\gamma}
\Phi_{\lambda,p}^{\gamma}(h_\lambda;Q_0),
\end{align*}
where the implicit positive constant depends only on $N$, $p$, and
$\gamma$.
Taking the infimum over admissible families, we obtain
$\kappa_{N,p,\gamma}\in(0,\infty]$.

Taking $h_\lambda=h_0$ in \eqref{6-1-0} and using
\eqref{aff:eq:affine-limit} with $e=\frac{(1,\ldots,1)}{\sqrt N}$,
we obtain
\begin{equation*}
\kappa_{N,p,\gamma}
\leq\lim_{\lambda\to L_\gamma}
\Phi_{\lambda,p}^{\gamma}(h_0;Q_0)
=\frac{K_{N,p}}{|\gamma|}.
\end{equation*}
This completes the proof of Corollary~\ref{cor:kappa-all-negative}.
\end{proof}

Using Corollary \ref{cor:kappa-all-negative}, together
with Propositions \ref{prop:positive-gamma-global}, \ref{prop:minus-one-sharp-global},
and \ref{prop:constant-negative-infinity}, we next show Corollary \ref{thm:classification}.

\begin{proof}[Proof of Corollary \ref{thm:classification}]
Propositions \ref{prop:positive-gamma-global} and \ref{prop:minus-one-sharp-global}
give (i), while Corollary~\ref{thm:unified-negative} yields (ii).
For any $M\in(1,\infty)$ and $\gamma\in[-M,-1]$, the estimate
$8^{|\gamma|}\leq8^M$ gives the upper bound in (iii).
From Proposition \ref{prop:constant-negative-infinity}, we infer that
the lower bound in (iii) holds.
Finally, Corollary \ref{cor:kappa-all-negative} gives (iv).
This completes the proof of Corollary \ref{thm:classification}.
\end{proof}

\subsection{The Classical Poincar\'e Inequality}\label{s3.5}

In this subsection, we investigate how
inequality \eqref{1-1}
recovers the classical Poincar\'e inequality on cubes.

Let $p\in[1,\infty)$ and $\gamma\in(-\infty,0)$.
Let $Q\subset\mathbb R^N$ be a cube and $g\in C^1(\overline Q)$.
Repeating the argument used in the proof of \cite[Lemma~3.3(ii)]{BSVY},
with both variables $x$ and $y$ restricted to $Q$ and the modulus of
continuity of $\nabla u$ on $\mathbb R^N$ replaced by that of
$\nabla g$ on $\overline Q$, we obtain
\begin{equation}\label{app:eq:smooth-local-limit}
\varlimsup_{\lambda\to0^+}\Phi_{\lambda,p}^{\gamma}(g;Q)
\leq\frac{K_{N,p}}{|\gamma|}\int_Q|\nabla g|^p\,dx.
\end{equation}
Here, the convexity of $Q$ ensures that the line segments used in that
argument remain in $Q$.

We now explain how Corollary~\ref{thm:classification}(ii) recovers the
classical Poincar\'e inequality. Let $p\in[1,\infty)$ and
$\gamma\in(-\infty,-1]$. For $g\in C^1(\overline{Q_0})$, letting
$\lambda\to0^+$ in \eqref{1-1} and applying \eqref{app:eq:smooth-local-limit} yield
\begin{equation}\label{app:eq:classical-poincare}
\int_{Q_0}\int_{Q_0}|g(x)-g(y)|^p\,dx\,dy
\lesssim\int_{Q_0}|\nabla g|^p\,dx,
\end{equation}
where the implicit positive constant depends only on $N$, $p$, and $\gamma$.
By the density of $C^\infty(\overline{Q_0})$ in $W^{1,p}(Q_0)$
(see, for instance, \cite[Theorem~11.35]{Leoni}), we obtain
\eqref{app:eq:classical-poincare} for any $g\in W^{1,p}(Q_0)$.

\section{Characterizations of Sobolev and BV Spaces}
\label{sec:gamma-applications}

The purpose of this section is to present a unified proof of known
$\Gamma$-liminf lower bounds and liminf characterizations of Sobolev
and $\mathrm{BV}$ spaces for negative and positive parameters.
We recall the relevant results of Nguyen \cite{Nguyen2007,Nguyen2010,NguyenSurvey},
Gobbino and Picenni \cite{GobbinoPicenni},
Lin et al.\ \cite{LinYangYangYuanZhang,LinEtAlFullRange2026}, and Picenni \cite{Picenni2026}.
Corollary~\ref{app:thm:gamma-coercivity} follows directly from
Theorem~\ref{thm:truncated-differences}, and we then apply it to constant
sequences. The notation $\lambda\to L_\gamma$ is as in the introduction.

\begin{proof}[Proof of Corollary \ref{app:thm:gamma-coercivity}]
Let $\{\lambda_j\}_{j\in\mathbb N}\subset(0,\infty)$ and
$\{u_j\}_{j\in\mathbb N}\subset L^1_{\mathrm{loc}}(\Omega)$ satisfy
$\lambda_j\to L_\gamma$ and $u_j\to u$ in
$L^1_{\mathrm{loc}}(\Omega)$.
If the liminf on the left-hand side is infinite, the assertion is immediate.
For $y,z\in\Omega$ with $y\ne z$, define
$$
F_j(y,z):=
\begin{cases}
\left(|u_j(y)-u_j(z)|
-\lambda_j|y-z|^{1+\frac{\gamma}{p}}\right)_+
&\text{if }\gamma\in(-\infty,-1],\\
\min\left\{|u_j(y)-u_j(z)|,
\lambda_j|y-z|^{1+\frac{\gamma}{p}}\right\}
&\text{if }\gamma\in(0,\infty),
\end{cases}
$$
and we define $F_j(y,y):=0$.
Let $U$ be any bounded open set with $\overline U\subset\Omega$, and
let $\widetilde F_j$ be defined as $F_j$ with $u_j$ replaced by $u$.
For any $a\in[0,\infty)$, the maps $s\mapsto(s-a)_+$ and
$s\mapsto\min\{s,a\}$ are $1$-Lipschitz on $[0,\infty)$. Thus,
as $j\to\infty$,
$$
\left\|F_j-\widetilde F_j\right\|_{L^1(U\times U)}
\leq2|U|\,\|u_j-u\|_{L^1(U)}\to0.
$$
Moreover, the prescribed limit of $\lambda_j$ gives
$\widetilde F_j(y,z)\to|u(y)-u(z)|$ for almost every
$(y,z)\in U\times U$,
$0\leq\widetilde F_j(y,z)\leq|u(y)-u(z)|,$
and
$$
\int_U\int_U|u(y)-u(z)|\,dy\,dz
\leq2|U|\,\|u\|_{L^1(U)}<\infty.
$$
By these estimates and the dominated convergence theorem,
we conclude that
\begin{equation}\label{eq:truncations-recover-differences}
F_j\to|u(y)-u(z)|
\quad\text{in }L^1(U\times U).
\end{equation}
Let $\Omega'$ be a bounded open set with
$\overline{\Omega'}\subset\Omega$, and choose a bounded open set $U$
such that $\overline{\Omega'}\subset U$ and $\overline U\subset\Omega$.
Fix a nonnegative $\rho\in C_{\rm c}^\infty((-1,1)^N)$ with
$\int\rho\,dx=1$.
For sufficiently small $\varepsilon\in(0,\infty)$, let
$\rho_\varepsilon(x):=\varepsilon^{-N}\rho(\varepsilon^{-1}x)$
and $R_\varepsilon(x):=x+(-\varepsilon,\varepsilon)^N$
so that $R_\varepsilon(x)\subset U$ for $x\in\Omega'$.
From \eqref{eq:truncations-recover-differences}, we deduce that, for
any fixed such $\varepsilon$,
\begin{align*}
&\sup_{x\in\Omega'}
\left|
\fint_{R_\varepsilon(x)}\fint_{R_\varepsilon(x)}
\left[F_j(y,z)-|u(y)-u(z)|\right]\,dy\,dz
\right|\\
&\qquad\leq(2\varepsilon)^{-2N}
\left\|F_j-|u(y)-u(z)|\right\|_{L^1(U\times U)}
\to0
\end{align*} as $j\to\infty$.
Since $\int\nabla\rho_\varepsilon\,dx={\bf0}$, we infer that
$$
|\nabla(\rho_\varepsilon*u)(x)|
\lesssim\varepsilon^{-1}
\fint_{R_\varepsilon(x)}\fint_{R_\varepsilon(x)}
|u(y)-u(z)|\,dy\,dz.
$$
Using these estimates, the triangle inequality in $L^p(\Omega')$, and
Theorem~\ref{thm:truncated-differences}, we obtain
\begin{align*}
&\int_{\Omega'}|\nabla(\rho_\varepsilon*u)|^p\,dx\\
&\quad\lesssim\varepsilon^{-p}\lim_{j\to\infty}
\int_{\Omega'}
\left[
\fint_{R_\varepsilon(x)}\fint_{R_\varepsilon(x)}
F_j(y,z)\,dy\,dz
\right]^p\,dx\\
&\quad\leq\varepsilon^{-p}\varliminf_{j\to\infty}
\int_{\Omega'}\fint_{R_\varepsilon(x)}\fint_{R_\varepsilon(x)}
F_j(y,z)^p\,dy\,dz\,dx\\
&\quad\lesssim\varepsilon^{-N}\varliminf_{j\to\infty}
\int_{\Omega'}\Phi_{\lambda_j,p}^{\gamma}
(u_j;R_\varepsilon(x))\,dx\lesssim\varliminf_{j\to\infty}
\Phi_{\lambda_j,p}^{\gamma}(u_j;\Omega),
\end{align*}
where the last step used Fubini's theorem and the fact that, for any
fixed $y,z\in\Omega$, the set of centers $x$ satisfying
$y,z\in R_\varepsilon(x)$ has measure at most $(2\varepsilon)^N$,
and where the implicit positive constants depend only on $N$, $p$, and
$\gamma$.
By the convergence $\rho_\varepsilon*u\to u$ in $L^1(\Omega')$, the
weak compactness of the gradients and the weak lower semicontinuity of the
$L^p$ norm when $p\in(1,\infty)$
(see \cite[Corollary~A.65 and Proposition~A.67]{Leoni}), or the lower
semicontinuity of the total variation when $p=1$
(see \cite[Exercise~14.3]{Leoni}), we conclude that, as $\varepsilon\to0^+$,
$$
\mathcal E_p(u;\Omega')
\lesssim\varliminf_{j\to\infty}
\Phi_{\lambda_j,p}^{\gamma}(u_j;\Omega),
$$
where the implicit positive constant is independent of $\Omega'$.
Exhausting $\Omega$ by relatively compact open sets gives
the desired conclusion. This completes the proof of
Corollary~\ref{app:thm:gamma-coercivity}.
\end{proof}

Applying Corollary~\ref{app:thm:gamma-coercivity} to a constant
sequence and using the upper bounds in
\cite[Theorems~1.3(i) and 1.4(i)]{BSVY}, we immediately
obtain the following
liminf characterizations on $\mathbb R^N$.

\begin{corollary}
Let $p\in[1,\infty)$, $\gamma\in\mathbb R\setminus(-1,0]$, and
$u\in L^1_{\mathrm{loc}}(\mathbb R^N)$. Then
\begin{equation}\label{app:eq:liminf-lower-bound}
\mathcal E_p(u;\mathbb R^N)
\lesssim\varliminf_{\lambda\to L_\gamma}\Phi_{\lambda,p}^\gamma(u;\mathbb R^N),
\end{equation}
where the implicit positive constant depends only on $N$, $p$, and $\gamma$.
In particular, if the right-hand side of \eqref{app:eq:liminf-lower-bound}
is finite, then $u\in\dot W^{1,p}(\mathbb R^N)$ when $p\in(1,\infty)$
and $u\in\dot{\mathrm{BV}}(\mathbb R^N)$ when $p=1$.
Moreover, if $(p,\gamma)\ne(1,-1)$, then
\begin{equation}\label{app:eq:liminf-characterization}
\mathcal E_p(u;\mathbb R^N)
\sim\varliminf_{\lambda\to L_\gamma}\Phi_{\lambda,p}^\gamma(u;\mathbb R^N)
\sim\sup_{\lambda\in(0,\infty)}\Phi_{\lambda,p}^\gamma(u;\mathbb R^N),
\end{equation}
where the positive equivalence constants depend only on $N$, $p$, and $\gamma$.
\end{corollary}

\begin{remark}
\begin{enumerate}[\rm(i)]
\item When $p=1$ and $\gamma=-1$, estimate \eqref{app:eq:liminf-lower-bound} still holds,
whereas the equivalences in \eqref{app:eq:liminf-characterization} fail.
Indeed, \cite[Theorem~1.8(ii)]{BSVY} provides a compactly supported
$u\in W^{1,1}(\mathbb R^N)$ for which
$\lim_{\lambda\to0^+}\Phi_{\lambda,1}^{-1}(u;\mathbb R^N)=\infty$.
Such pathological pointwise behaviour was already observed by Brezis and Nguyen
in \cite[Section~2.2, Pathology~1]{BrezisNguyenImages}.

\item For $\gamma\in(-\infty,-p)$, estimate
\eqref{app:eq:liminf-lower-bound} replaces the limsup in
\cite[Theorem~15]{NguyenSurvey} by the liminf, with an implicit positive
constant in place of the coefficient $\frac{K_{N,p}}{|\gamma|}$.
Recently, this liminf bound was also obtained in
\cite[Corollary~1.2]{Picenni2026}.
For $\gamma\in[-p,-1]$, estimate \eqref{app:eq:liminf-lower-bound} recovers the liminf estimate in
\cite[Theorem~16]{NguyenSurvey}.

\end{enumerate}
\end{remark}

\section{Further Applications}
\label{sec:applications}

The target of this section is to apply the Poincar\'e inequality to
establish some new
compactness results and the existence of
minimizers at a fixed threshold in Subsection \ref{s5.1},
and a criterion for vanishing mean oscillation in Subsection \ref{s5.2}.

\subsection{Compactness and Variational Problems}\label{s5.1}

In this subsection, motivated by \cite[Theorem~2]{Nguyen2011}
and \cite[Theorem~2 and Corollary~7]{BrezisNguyenImages},
we establish compactness results for thresholds tending to zero
and for a fixed threshold.
We then apply the latter result to prove the existence of minimizers
for a variational problem involving $\Phi_{\lambda,p}^\gamma$.

\begin{theorem}\label{app:thm:asymptotic-compactness}
Let $p\in[1,\infty)$,
$\gamma\in(-\infty,-1]$, and $\{\lambda_j\}_{j\in\mathbb N}\subset(0,\infty)$
be a sequence satisfying $\lambda_j\to0^+$ as $j\to\infty$.
Then the following statements hold.
\begin{enumerate}[\rm(i)]
\item Let $Q_0:=(0,1)^N$ and $\{u_j\}_{j\in\mathbb N}\subset L^1(Q_0)$. If
\begin{equation}\label{app:eq:compactness-assumption}
\sup_j\|u_j\|_{L^1(Q_0)}<\infty
\quad\text{and}\quad\varliminf_{j\to\infty}\Phi_{\lambda_j,p}^\gamma(u_j;Q_0)<\infty,
\end{equation}
then there exist a subsequence $\{u_{j_k}\}_{k\in\mathbb N}$ and $u\in L^p(Q_0)$
such that $u_{j_k}\to u$ in $L^p(Q_0)$ and
$$
\mathcal E_p(u;Q_0)
\lesssim\varliminf_{k\to\infty}
\Phi_{\lambda_{j_k},p}^\gamma(u_{j_k};Q_0)<\infty,
$$
where the implicit positive constant depends only on $N$, $p$, and $\gamma$.
In particular,
$u\in W^{1,p}(Q_0)$ when $p\in(1,\infty)$, or
$u\in\mathrm{BV}(Q_0)$ when $p=1$.
\item Let $\Omega\subset\mathbb R^N$ be an open set and $\{u_j\}_{j\in\mathbb N}\subset L^p(\Omega)$. If
\begin{equation}\label{app:eq:local-compactness-assumption}
\sup_j\|u_j\|_{L^p(\Omega)}<\infty
\quad\text{and}\quad
\varliminf_{j\to\infty}\Phi_{\lambda_j,p}^\gamma(u_j;\Omega)<\infty,
\end{equation}
then there exist a subsequence $\{u_{j_k}\}_{k\in\mathbb N}$ and $u\in L^p(\Omega)$
such that $u_{j_k}\to u$ in $L^p_{\mathrm{loc}}(\Omega)$ and
$$
\mathcal E_p(u;\Omega)
\lesssim\varliminf_{k\to\infty}\Phi_{\lambda_{j_k},p}^\gamma(u_{j_k};\Omega)
<\infty,
$$
where the implicit positive constant depends only on $N$, $p$, and $\gamma$.
In particular,
$u\in W^{1,p}(\Omega)$ when $p\in(1,\infty)$, or $u\in\mathrm{BV}(\Omega)$
when $p=1$.
\end{enumerate}
\end{theorem}

\begin{remark}
When $\gamma=-p$ and $\Omega=\mathbb R^N$, part (ii) of
Theorem~\ref{app:thm:asymptotic-compactness} in this case reduces
to Nguyen's nonlocal
Rellich--Kondrachov theorem \cite[Theorem~2]{Nguyen2011}.
\end{remark}

\begin{proof}[Proof of Theorem~\ref{app:thm:asymptotic-compactness}]
We first show (i). Passing to a subsequence, we may assume that
$\Phi_{\lambda_j,p}^\gamma(u_j;Q_0)$ converges to the finite liminf
in \eqref{app:eq:compactness-assumption}.
Applying \eqref{1-1} on $Q_0$
together with \eqref{app:eq:compactness-assumption}, we obtain a
uniform bound for $\|u_j\|_{L^p(Q_0)}$. For $r=2^{-k}$ with $k\in\mathbb N$,
let
$$
P_ru_j:=\sum_R(u_j)_R\mathbf1_R,
$$
where the sum is taken over $R:=2^{-k}[z+(0,1)^N]$ with $z\in\{0,\ldots,2^k-1\}^N$.
These cubes form a partition of $Q_0$ up to a set of measure zero.
Using \eqref{1-1} on these cubes and summing, we find that
\begin{align}\label{app:eq:finite-dimensional-approximation}
\|u_j-P_ru_j\|_{L^p(Q_0)}^p
&=\sum_R\int_R|u_j-(u_j)_R|^p\,dx\notag\\
&\lesssim \lambda_j^p r^{p+\gamma}
+r^p\Phi_{\lambda_j,p}^\gamma(u_j;Q_0),
\end{align}
where the implicit positive constant is independent of $j$ and $r$.
For fixed $k\in\mathbb N$, the finitely many averages $(u_j)_R$ are
bounded uniformly in $j$, so they converge simultaneously along a subsequence.
Repeating this selection for $k\in\{1,2,\ldots\}$ and taking the $j$th term
of the $j$th subsequence, we obtain a subsequence along which
$P_{2^{-k}}u_j$ converges in $L^p(Q_0)$ for any $k\in\mathbb N$.
We retain the notation $u_j$ for this subsequence.
From the triangle inequality and
\eqref{app:eq:finite-dimensional-approximation}, it follows that
\begin{align*}
\varlimsup_{i,j\to\infty}\|u_i-u_j\|_{L^p(Q_0)}^p
\lesssim\varlimsup_{j\to\infty}\|u_j-P_ru_j\|_{L^p(Q_0)}^p
\lesssim r^p\lim_{j\to\infty}\Phi_{\lambda_j,p}^\gamma(u_j;Q_0).
\end{align*}
Letting $r\to0^+$, we conclude that the selected subsequence is Cauchy
in $L^p(Q_0)$ and hence converges to a limit $u\in L^p(Q_0)$. Applying
Corollary~\ref{app:thm:gamma-coercivity} to this subsequence with
$\Omega=Q_0$, we obtain the asserted regularity of $u$.

We next prove (ii). Passing to a subsequence, we may assume that
$\Phi_{\lambda_j,p}^\gamma(u_j;\Omega)$ converges to the finite liminf
in \eqref{app:eq:local-compactness-assumption}.
Choose open cubes $\{R_m\}_{m\in\mathbb N}$ with
$\Omega=\bigcup_{m\in\mathbb N}R_m$, where
$\overline{R_m}\subset\Omega$ for any $m\in\mathbb N$.
Applying (i) after translation and dilation, and successively
extracting subsequences, we obtain a diagonal subsequence $\{u_{j_k}\}_{k\in\mathbb N}$
and functions $v_m\in L^p(R_m)$ such that
$u_{j_k}\to v_m$ in $L^p(R_m)$ for any $m\in\mathbb N$.
The limits agree almost everywhere on overlaps of these cubes.
Since the cover is countable, we deduce that, outside a common null set,
we can define $u(x):=v_m(x)$ whenever $x\in R_m$, independently of the
choice of $m$; on that null set, we define $u:=0$.
Since any compact subset of $\Omega$ is covered by finitely many $R_m$,
we infer that $u\in L^p_{\mathrm{loc}}(\Omega)$ and
$u_{j_k}\to u$ in $L^p_{\mathrm{loc}}(\Omega)$.
Passing to a further subsequence, we also obtain $u_{j_k}\to u$
almost everywhere in $\Omega$.
From Fatou's lemma and \eqref{app:eq:local-compactness-assumption},
it follows that
$$
\|u\|_{L^p(\Omega)}^p
\leq\varliminf_{k\to\infty}\|u_{j_k}\|_{L^p(\Omega)}^p<\infty.
$$
Applying Corollary~\ref{app:thm:gamma-coercivity} on any bounded
open set $U$ with $\overline U\subset\Omega$, we conclude that
$$
\mathcal E_p(u;U)
\lesssim\varliminf_{k\to\infty}\Phi_{\lambda_{j_k},p}^\gamma(u_{j_k};U)
\leq\varliminf_{k\to\infty}\Phi_{\lambda_{j_k},p}^\gamma(u_{j_k};\Omega),
$$
where the implicit positive constant is independent of $U$.
Exhausting $\Omega$ by such sets $U$, we obtain the asserted
lower bound and regularity of $u$. This completes the proof of
Theorem~\ref{app:thm:asymptotic-compactness}.
\end{proof}

Theorem~\ref{app:thm:asymptotic-compactness} concerns thresholds tending to zero.
For variational problems at a fixed threshold, compactness at that threshold
is a key ingredient in showing the existence of minimizers; see, for instance,
Brezis and Nguyen \cite[Theorem~2]{BrezisNguyenImages}.
The class of functionals covered by their fixed-parameter compactness theorem,
however, does not include the family $\{\Phi_{\lambda,p}^\gamma\}_{\lambda\in(0,\infty)}$
considered here. The following theorem identifies the precise range of $\gamma$
for which compactness holds at a fixed threshold.

\begin{theorem}\label{app:thm:fixed-compactness}
Let $p\in[1,\infty)$, $\gamma\in\mathbb R$,
$Q_0=(0,1)^N$, and $\lambda\in(0,\infty)$.
Then the following assertions are equivalent.
\begin{enumerate}[\rm(i)]
\item $\gamma\in(-p,-1]$.
\item Any sequence $\{u_j\}_{j\in\mathbb N}\subset L^1(Q_0)$ satisfying
$$
\sup_j\|u_j\|_{L^1(Q_0)}<\infty
\quad\text{and}\quad
\varliminf_{j\to\infty}\Phi_{\lambda,p}^\gamma(u_j;Q_0)<\infty
$$
has a subsequence converging in $L^p(Q_0)$.
\end{enumerate}
\end{theorem}

\begin{proof}
Assume first that $\gamma\in(-p,-1]$. Passing to a subsequence,
we may assume that $\sup_j\Phi_{\lambda,p}^\gamma(u_j;Q_0)<\infty$.
Since $p+\gamma\in(0,\infty)$, we deduce that the right-hand side of
\eqref{app:eq:finite-dimensional-approximation} with $\lambda_j=\lambda$
tends to zero as $r\to0^+$ uniformly in $j$.
Repeating the compactness argument in the proof of
Theorem~\ref{app:thm:asymptotic-compactness}(i), we obtain a subsequence
converging in $L^p(Q_0)$.

To prove necessity, first assume that $\gamma\leq-p$ and write
$\beta=1+\frac{\gamma}{p}\leq0$. Choose
$a\in(0,\lambda(\sqrt N)^\beta]$, and define $\{u_j\}_{j\in\mathbb N}$ by
$$
u_j(x):=a\mathbf1_{\{\lfloor2^jx_1\rfloor\in2\mathbb Z\}}(x),
\quad \forall\,x:=(x_1,\ldots,x_N)\in Q_0,
$$
where $\lfloor t\rfloor$ denotes the largest integer not greater than $t$.
Then $|u_j(x)-u_j(y)|\leq\lambda|x-y|^\beta$ for $x\ne y$, so that
$\Phi_{\lambda,p}^\gamma(u_j;Q_0)=0$. On the other hand,
$\|u_j-u_k\|_{L^p(Q_0)}^p=\frac{a^p}{2}$ for $j\ne k$,
and no subsequence converges in $L^p(Q_0)$.

Next, assume that $\gamma\in(-1,\infty)$. Let $x_0$ be the center of $Q_0$,
choose a sequence $\{\varepsilon_j\}_{j\in\mathbb N}\subset(0,\frac{1}{4})$
with $\varepsilon_j\to0^+$, and define $\{u_j\}_{j\in\mathbb N}$ by
$$
u_j:=|B(x_0,\varepsilon_j)|^{-\frac{1}{p}}
\mathbf1_{B(x_0,\varepsilon_j)}.
$$
Then $\|u_j\|_{L^p(Q_0)}=1$, $\|u_j\|_{L^1(Q_0)}\leq1$, and
$u_j\to0$ almost everywhere. To estimate $\Phi_{\lambda,p}^\gamma(u_j;Q_0)$,
observe that, for any $h\in\mathbb R^N$,
$$
\left|B(x_0,\varepsilon_j)\mathbin{\triangle}
B(x_0-h,\varepsilon_j)\right|
\lesssim\min\{\varepsilon_j^N,\varepsilon_j^{N-1}|h|\},
$$
where $A\mathbin{\triangle}B:=(A\setminus B)\cup(B\setminus A)$
denotes the symmetric difference.
Dropping the threshold condition while retaining $u_j(x)\ne u_j(y)$,
and making the change of variables $h:=y-x$, we obtain
\begin{align*}
\Phi_{\lambda,p}^\gamma(u_j;Q_0)
&\leq\lambda^p\int_{Q_0}\int_{Q_0} |x-y|^{\gamma-N}
\mathbf1_{\{u_j(x)\ne u_j(y)\}}\,dx\,dy\\
&\leq\lambda^p\int_{|h|\leq\sqrt N}|h|^{\gamma-N}
\left|B(x_0,\varepsilon_j)\mathbin{\triangle}
B(x_0-h,\varepsilon_j)\right|\,dh.
\end{align*}
Using the symmetric difference estimate above and polar coordinates,
we find that
\begin{align*}
\Phi_{\lambda,p}^\gamma(u_j;Q_0)
&\lesssim\lambda^p\int_0^{\sqrt N}r^{\gamma-1}
\min\{\varepsilon_j^N,\varepsilon_j^{N-1}r\}\,dr\\
&=\lambda^p\left[
\varepsilon_j^{N-1}\int_0^{\varepsilon_j}r^\gamma\,dr
+\varepsilon_j^N\int_{\varepsilon_j}^{\sqrt N}
r^{\gamma-1}\,dr\right]\\
&\lesssim\lambda^p
\begin{cases}
\varepsilon_j^{N+\gamma}&\text{if }\gamma\in(-1,0),\\
\varepsilon_j^N(1+|\log\varepsilon_j|)&\text{if }\gamma=0,\\
\varepsilon_j^N&\text{if }\gamma\in(0,\infty),
\end{cases}
\end{align*}
where the implicit positive constants are independent of $j$.
Thus, $\Phi_{\lambda,p}^\gamma(u_j;Q_0)\to0$ as $j\to\infty$, while
no subsequence converges in $L^p(Q_0)$.
In particular, the compactness range
in Theorem~\ref{app:thm:fixed-compactness} is empty when $p=1$.
This completes the proof of Theorem \ref{app:thm:fixed-compactness}.
\end{proof}

We next consider a variational problem associated with $\Phi_{\lambda,p}^\gamma$.
Let $p\in(1,\infty)$, $\gamma\in(-p,-1]$, and
$Q_0=(0,1)^N$. For $\lambda\in(0,\infty)$ and $f\in L^p(Q_0)$, define
$$
J_{\lambda,f}(u):=\Phi_{\lambda,p}^\gamma(u;Q_0)+\|u-f\|_{L^p(Q_0)}^p,
\quad \forall\,u\in L^p(Q_0),
$$
and consider the minimization problem
$m_\lambda(f):=\inf_{u\in L^p(Q_0)}J_{\lambda,f}(u)$.

\begin{corollary}\label{app:cor:regularization}
Under the above assumptions, for any $\lambda\in(0,\infty)$,
there exists $u\in L^p(Q_0)$ such that
$J_{\lambda,f}(u)=m_\lambda(f)$.
Moreover, any sequence $\{u_j\}_{j\in\mathbb N}\subset L^p(Q_0)$
satisfying $J_{\lambda,f}(u_j)\to m_\lambda(f)$ as $j\to\infty$ is relatively
compact in $L^p(Q_0)$, and any limit point is a minimizer.
\end{corollary}

\begin{proof}
Using Fatou's lemma, we easily
conclude that $\Phi_{\lambda,p}^\gamma(\cdot;Q_0)$ is lower
semicontinuous in $L^p(Q_0)$; that is,
$$
\Phi_{\lambda,p}^\gamma(u;Q_0)
\leq\varliminf_{j\to\infty}\Phi_{\lambda,p}^\gamma(u_j;Q_0)
\quad\text{whenever }u_j\to u\text{ in }L^p(Q_0).
$$
Let $\{u_j\}_{j\in\mathbb N}$ be a minimizing sequence. Since
$J_{\lambda,f}(u_j)\to m_\lambda(f)\leq
J_{\lambda,f}(0)=\|f\|_{L^p(Q_0)}^p<\infty$,
we infer that, after discarding finitely many terms,
$
\sup_j J_{\lambda,f}(u_j)<\infty.
$
By this and the triangle inequality, we obtain
$$\sup_j\|u_j\|_{L^1(Q_0)}
\leq\sup_j\|u_j\|_{L^p(Q_0)}
\leq\sup_j\|u_j-f\|_{L^p(Q_0)}+\|f\|_{L^p(Q_0)}<\infty.
$$
From Theorem~\ref{app:thm:fixed-compactness} and the fact that any
subsequence of $\{u_j\}_{j\in\mathbb N}$ satisfies the same bounds,
it follows that any such subsequence has a further subsequence converging in $L^p(Q_0)$. Thus, $\{u_j\}_{j\in\mathbb N}$
is relatively compact in $L^p(Q_0)$.
For any limit point $u$, choose a subsequence satisfying
$u_{j_k}\to u$ in $L^p(Q_0)$.
From the lower semicontinuity of $\Phi^\gamma_{\lambda,p}(\cdot;Q_0)$
and the continuity of
$u\mapsto\|u-f\|_{L^p(Q_0)}^p$, we deduce that
$$
m_\lambda(f)\leq J_{\lambda,f}(u)
\leq\lim_{k\to\infty}J_{\lambda,f}(u_{j_k})=m_\lambda(f),
$$
and hence $u$ is a minimizer.
This completes the proof of Corollary \ref{app:cor:regularization}.
\end{proof}

\subsection{A VMO Criterion at a Single Threshold}\label{s5.2}

In this subsection, we take $p=N$ and determine the range of $\gamma$
for which $\Phi_{\lambda,N}^\gamma(u;Q_0)<\infty$ at a single threshold
$\lambda\in(0,\infty)$ implies $u\in\mathrm{VMO}(Q_0)$.

For $u\in L^1(Q_0)$, define
\begin{equation*}
\|u\|_{\mathrm{BMO}(Q_0)}
:=\sup_{R\subset Q_0}\fint_R|u(x)-u_R|\,dx\quad
\text{and}\quad u_R:=\fint_Ru(x)\,dx,
\end{equation*}
where the supremum is taken over all cubes $R\subset Q_0$.
We say that $u\in\mathrm{BMO}(Q_0)$ if this seminorm is finite.
We say that $u\in\mathrm{VMO}(Q_0)$
if $u\in L^1(Q_0)$ and
$$
\lim_{r\to0^+}\sup_{\genfrac{}{}{0pt}{}{R\subset Q_0}{\ell(R)\leq r}}
\fint_R|u(x)-u_R|\,dx=0,
$$
where the supremum is taken over all cubes $R\subset Q_0$ with $\ell(R)\le r$.
From \cite[Proposition~1]{BrezisNguyenVMO}, we infer that
the following cube versions of the criteria of Brezis and Nguyen
\cite[Theorem~1]{BrezisNguyenVMO} hold for any $u\in L^1(Q_0)$:
\begin{align*}
\Phi_{\lambda,N}^{-N}(u;Q_0)<\infty
\quad\text{for some }\lambda\in(0,\infty)
&\quad\Longrightarrow\quad u\in\mathrm{BMO}(Q_0)
\end{align*}
and
\begin{align*}
\Phi_{\lambda,N}^{-N}(u;Q_0)<\infty
\quad\text{for all }\lambda\in(0,\infty)
&\quad\Longrightarrow\quad u\in\mathrm{VMO}(Q_0).
\end{align*}
This corresponds to the case $\gamma=-N$.
For $\gamma\in(-N,-1]$, any $u\in L^1(Q_0)$ satisfying
$\Phi_{\lambda,N}^{\gamma}(u;Q_0)<\infty$ for some
$\lambda\in(0,\infty)$ belongs to $\mathrm{VMO}(Q_0)$.
The following theorem shows that this range of $\gamma$ is sharp.
\begin{theorem}\label{app:thm:single-threshold-vmo}
Let $N\ge 2$ be an integer and $\gamma\in\mathbb R$.
Then any $u\in L^1(Q_0)$ satisfying
\begin{equation*}
\Phi_{\lambda,N}^\gamma(u;Q_0)<\infty
\quad\text{for some }\lambda\in(0,\infty)
\end{equation*}
belongs to $\mathrm{VMO}(Q_0)$ if and only if
$\gamma\in(-N,-1]$.
\end{theorem}

\begin{proof}
Assume first that $\gamma\in(-N,-1]$ and
$\Phi_{\lambda,N}^\gamma(u;Q_0)<\infty$ for $u\in L^1(Q_0)$ and some fixed
$\lambda\in(0,\infty)$. For $t\in(0,\infty)$, define
$$
\eta(t):=\lambda^N
\iint\nolimits_{E_{\lambda,\frac{\gamma}{N}}[u;Q_0]\cap\{|x-y|\leq t\}}
|x-y|^{\gamma-N}\,dx\,dy.
$$
By the dominated convergence theorem, we conclude that $\eta(t)\to0$ as $t\to0^+$.
Applying \eqref{1-1} to a cube $R\subset Q_0$
with $\ell(R)\leq r$, we obtain
\begin{align*}
\fint_R|u-u_R|^N\,dx
\lesssim\lambda^N\ell(R)^{N+\gamma}
+\Phi_{\lambda,N}^\gamma(u;R)\leq\lambda^N r^{N+\gamma}+\eta(\sqrt N r),
\end{align*}
where the last step used $N+\gamma\in(0,\infty)$ and
$|x-y|\leq\sqrt N r$ for $x,y\in R$, and where
the implicit positive constant is independent of
$R$ and $r$.
The right-hand side tends to zero independently of $R$, which gives
$u\in\mathrm{VMO}(Q_0)$.

To prove necessity, let $\lambda\in(0,\infty)$ and define
$u(x):=a\mathbf1_{\{x_1>\frac{1}{2}\}}(x)$ for $x:=(x_1,\ldots,x_N)\in Q_0$,
where $a\in(0,\infty)$ will be chosen below. This function does not
belong to $\mathrm{VMO}(Q_0)$ because its mean oscillation on any sufficiently small
cube centered at $\frac{1}{2}(1,\ldots,1)$ equals $\frac{a}{2}$.
If $\gamma\leq-N$, then $\beta:=1+\frac{\gamma}{N}\leq0$.
Choosing $a\in(0,\lambda(\sqrt N)^\beta]$, we obtain
$$
|u(x)-u(y)|\leq a\leq\lambda|x-y|^\beta,
\quad \forall\,x,y\in Q_0,\quad x\ne y,
$$
and hence $\Phi_{\lambda,N}^\gamma(u;Q_0)=0$.

If $\gamma\in(-1,\infty)$, any $a\in(0,\infty)$ gives a
counterexample. Indeed, for a fixed $h:=(h_1,\ldots,h_N)\in\mathbb R^N$, the set of
$x\in Q_0$ satisfying $x+h\in Q_0$ and $u(x+h)\ne u(x)$ has measure
at most $|h_1|$. Thus, by the change of variables $y=x+h$ and polar
coordinates, we obtain
\begin{align*}
\Phi_{\lambda,N}^\gamma(u;Q_0)
&\leq\lambda^N\iint\nolimits_{\{(x,y)\in Q_0^2:\ x\ne y,\ u(x)\ne u(y)\}}
|x-y|^{\gamma-N}\,dx\,dy\\
&\lesssim\lambda^N\int_0^{\sqrt N}r^\gamma\,dr<\infty.
\end{align*}
This completes the proof of Theorem \ref{app:thm:single-threshold-vmo}.
\end{proof}

\smallskip
\noindent\textbf{Acknowledgements}\quad
In the original version of this article, we used different methods
to show that the Poincar\'e inequality \eqref{1-1} holds in the ranges
$\gamma\in[-p,-1]$ and $\gamma\in(-\infty,-p)$.
After reading the article \cite{Frank2026} of Frank, we realized that our
method for $\gamma\in(-\infty,-p)$ could be combined with some ideas
from the proof of \cite[Lemma 2]{Frank2026} to unify our arguments for any $\gamma\in(-\infty,-1]$
and even extend to any $\gamma\in(0,\infty)$.
Developing this observation led us to the complementary truncation difference
estimates and the unified treatment of
$\gamma\in\mathbb R\setminus(-1,0]$ presented in this article.
We are very grateful to Professor Rupert L. Frank for reading our original
manuscript and for his help.

We would also like to mention the preprint
\emph{Lower bounds and Poincar\'e inequality for the Brezis--Seeger--Van Schaftingen--Yung functionals}
by Nicola Picenni \cite{Picenni2026}, in which closely related
Poincar\'e inequalities and $\Gamma$-liminf estimates were obtained
independently and almost simultaneously with our work, using different methods.
We were unaware of this work when the original version of this article was completed.
For $\gamma\in(-\infty,-p)$, Picenni
\cite[Theorems~1.3 and~1.1]{Picenni2026} proved a Poincar\'e inequality
on general bounded convex open sets and a positive $\Gamma$-liminf lower bound
in $L^1_{\mathrm{loc}}(\Omega)$ on any open set $\Omega$.
Our unified argument yields the Poincar\'e inequality on cubes for
$\gamma\in(-\infty,-1]$ and the $\Gamma$-liminf lower bound on any open set for
$\gamma\in\mathbb R\setminus(-1,0]$
(Corollaries~\ref{thm:classification} and~\ref{app:thm:gamma-coercivity}).

Yirui Zhao would like to thank Dr. Xiaosheng Lin for his helpful discussion
and also to thank Dr. Yinqin Li for drawing his attention
to the topic of variants of Poincar\'e inequalities.

The authors acknowledge the use of AI tools during the exploratory stage
of this project. All mathematical arguments and proofs in the final
manuscript were checked and
written by the authors.

\bigskip
\noindent\begin{minipage}{\linewidth}
\noindent Feng Dai
\smallskip

\noindent Department of Mathematical and Statistical Sciences,
University of Alberta, Edmonton, Alberta T6G 2G1, Canada
\smallskip

\noindent\textit{E-mail:} \texttt{fdai@ualberta.ca}
\end{minipage}

\bigskip
\noindent\begin{minipage}{\linewidth}
\noindent Dachun Yang, Wen Yuan and Yirui Zhao
\medskip

\noindent Laboratory of Mathematics and Complex Systems
(Ministry of Education of China), School of Mathematical Sciences,
Institute for Advanced Study, Beijing Normal University,
Beijing 100875, The People's Republic of China
\smallskip

\noindent\textit{E-mails:} \texttt{dcyang@bnu.edu.cn} (D. Yang)

\noindent\phantom{\textit{E-mails:}} \texttt{wenyuan@bnu.edu.cn} (W. Yuan)

\noindent\phantom{\textit{E-mails:}} \texttt{yiruizhao@mail.bnu.edu.cn} (Y. Zhao)
\end{minipage}

\end{document}